\documentclass{amsart}
\usepackage{amssymb,graphicx}
\usepackage{mathrsfs}
\usepackage{color}
\usepackage{enumerate}
\usepackage{a4wide}
\usepackage[colorlinks=true,citecolor=blue,linkcolor=blue]{hyperref}

\newcommand{\Rl}{\mathbb{R}}
\newcommand{\Cplx}{\mathbb{C}}
\newcommand{\Itgr}{\mathbb{Z}}
\newcommand{\Ntrl}{\mathbb{N}}
\newcommand{\Circ}{\mathbb{T}}

\newcommand{\Bc}{\mathcal{B}}

\newcommand{\Lc}{\mathcal{L}}
\newcommand{\Mv}{\mathcal{M}}

\newcommand{\Sc}{\mathcal{S}}

\newcommand{\Tr}{\mathrm{Tr}}

\newcommand{\res}{\mathrm{R}}
\newcommand{\pilo}{\Pi_{\ell}}
\newcommand{\pihi}{\Pi_{h}}

\newcommand{\supp}{\mathrm{supp}}
\newcommand{\BS}{\mathfrak{BS}}

\newcommand{\Ti}{\mathcal{T}}

\newcommand{\lt}{\lambda_\theta}
\newcommand{\ri}{{\mathrm{i}}}
\newcommand{\wl}{W_\theta}
\newcommand{\rank}{\mathrm{rank}}

\newcommand{\Pl}{\mathbb{P}}

\newcommand{\at}{\ast_\theta}
\newcommand{\rl}{\mathrm{Re}}

\def\XXint#1#2#3{{\setbox0=\hbox{$#1{#2#3}{\int}$ }
\vcenter{\hbox{$#2#3$ }}\kern-.6\wd0}}

\numberwithin{equation}{section}

\newtheorem{theorem}{Theorem}[section]
\newtheorem{proposition}[theorem]{Proposition}
\newtheorem{corollary}[theorem]{Corollary}
\newtheorem{definition}[theorem]{Definition}
\newtheorem{lemma}[theorem]{Lemma}

\newtheorem*{theorem*}{Theorem}

\theoremstyle{remark}

\newtheorem{remark}[theorem]{Remark}

\title{Nonlinear partial differential equations on noncommutative Euclidean spaces}
\author{E. McDonald}
\address{Edward McDonald, Department of Mathematics, Penn State University, University Park, PA 16801, USA}
\email{eam6282@psu.edu}
\date{\today}

\begin{document}
\maketitle{}
\begin{abstract}
Noncommutative Euclidean spaces -- otherwise known as Moyal spaces or quantum Euclidean spaces -- are a standard example
of a non-compact noncommutative geometry. Recent progress in the harmonic analysis of these spaces
gives us the opportunity to highlight some of their peculiar features. For example, the theory of nonlinear partial differential equations
has unexpected properties in this noncommutative setting. We develop elementary aspects of paradifferential calculus for noncommutative Euclidean spaces
and give some applications to nonlinear evolution equations. We demonstrate how the analysis of some equations radically simplifies
in the strictly noncommutative setting.
\end{abstract}

\section{Introduction}
    Noncommutative Euclidean $d$-spaces $\Rl^d_\theta$ are deformation quantizations of Euclidean space $\Rl^d$ in the sense of Rieffel \cite{Rieffel-memoirs}. 
    This family of spaces is one of the oldest and best studied in noncommutative geometry, and has received particular attention due to its relevance to the phase space picture of quantum mechanics.
    Many authors, including Moyal \cite{Moyal} and Groenwald \cite{Groen}, have studied these spaces from diverse perspectives. The main idea is to deform the algebra of smooth functions
    on Euclidean space $\Rl^d$ by replacing the pointwise product of functions with the twisted Moyal product. 
    
    In noncommutative geometry, noncommutative Euclidean spaces are a noteworthy example of ``noncompact" (or nonunital) spaces \cite{gayral-moyal,CGRS2}.
    In the mathematical physics literature, the Moyal product is studied for its relevance to quantum phase space \cite[Chapter 13]{Hall2013}, \cite[Chapter 2, Section 3.4]{Takhtajan2008}, \cite[Section 5.2.2.2]{Bratteli-Robinson2}, \cite{GV1988}. Besides this, noncommutative Euclidean spaces are studied in physics as a prototypical setting with noncommuting spatial coordinates, we note in particular \cite{DouglasNekrasov,NekrasovSchwarz}. 
    
    Several equivalent constructions relating to noncommutative Euclidean space exist in the literature. One method to describe $\Rl^d_\theta$ begins by defining the von Neumann algebra $L_\infty(\Rl^d_\theta)$ as a twisted left-regular representation of $\Rl^d$ on $L_2(\Rl^d)$; we will review this definition of $L_\infty(\Rl^d_\theta)$ in Subsection \ref{definition_subsection} below. 
    As a von Neumann algebra, $L_\infty(\Rl^d_\theta)$ is generated by a $d$-parameter strongly continuous unitary
    family $\{\lt(t)\}_{t \in \Rl^d}$ satisfying the relation
    $$
        \lt(t)\lt(s) = \exp(\frac12 \ri(t,\theta s))\lt(t+s),\quad t,s\in \Rl^d.
    $$
    With $\theta=0$, this reduces to the description of $L_\infty(\Rl^d)$ as being a von Neumann algebra generated by the family of trigonometric functions $\lambda_0(t)(\xi) = \exp(\ri(t,\xi))$, $t,\xi\in \Rl^d$. This is the Weyl form of the canonical commutation relations and the structure of $L_{\infty}(\Rl^d_\theta)$ is determined by the Stone-von Neumann theorem.
    A related object is the heavily studied noncommutative torus $\Circ^d_\theta,$ as $L_{\infty}(\Circ^d_\theta)$ is isomorphic to the subalgebra of $L_\infty(\Rl^d_\theta)$ generated by $\{\lt(n)\}_{n\in \Itgr^d}$.
    
    Associated to the algebra $L_\infty(\Rl^d_\theta)$ is a semifinite normal trace functional $\tau$, which reduces to the Lebesgue integral when $\theta=0$ and gives us a notion
    of integration for $\Rl^d_\theta.$ As a special case of the theory of $L_p$-spaces associated to semifinite traces on von Neumann algebras, we can form the spaces $L_p(L_\infty(\Rl^d_\theta),\tau)$, which we abbreviate as $L_p(\Rl^d_\theta)$. These spaces reduce to classical $L_p$-spaces on Euclidean space when $\theta=0$, and when
    $\theta\neq 0$ are spaces of operators affiliated to $L_\infty(\Rl^d_\theta).$ 
    
    A feature of $L_\infty(\Rl^d_\theta)$ is that the nature of the trace $\tau$ changes dramatically depending on the dimension of the kernel of $\theta.$ 
    Due to the Stone-von Neumann theorem, when the matrix $\theta$ is non-degenerate (i.e., $\det(\theta)\neq 0$), the algebra $L_\infty(\Rl^d_\theta)$
    is isomorphic to the algebra of all bounded linear operators on $L_2(\Rl^{d/2})$. Under this isomorphism, $\tau$ is precisely a scalar multiple of the classical operator trace $\Tr$
    and the $L_p$-spaces $L_p(\Rl^d_\theta)$ coincide with the Schatten-von Neumann $\Lc_p$-spaces of compact operators. It follows that when $\det(\theta)\neq 0$ there is a continuous inclusion
    \begin{equation*}
        L_p(\Rl^d_\theta) \subseteq L_q(\Rl^d_\theta),\quad 1 \leq p \leq q\leq \infty.
    \end{equation*}
    In particular, and in marked contrast to the familiar situation on $\Rl^d$, we have a norm inequality
    \begin{equation}\label{automatic_boundedness}
        \det(2\pi\theta)^{-1/4}\|u\|_{L_\infty(\Rl^d_\theta)} \leq \|u\|_{L_2(\Rl^d_\theta)},\quad u \in L_2(\Rl^d_\theta).
    \end{equation}
    The factor $\det(2\pi \theta)^{-1/4}$ depends on the choice of normalisation for the trace $\tau$, but it necessarily diverges as $\theta$ becomes singular.
    In many ways, the measure theory of $L_\infty(\Rl^d_\theta)$ in the non-degenerate case bears more resemblance
    to a discrete set than $\Rl^d$. This point of view is further emphasised by the observation that the range of the trace $\tau$ on 
    projections is discrete, and in our normalisation consists of non-negative integer multiples of $\det(2\pi\theta)^{1/2}$.
    
    There is a canonical notion of differential calculus for $\Rl^d_\theta$, as well as Sobolev spaces and differential operators.    
    In this language, one can formulate partial differential equations associated to $\Rl^d_\theta$ and study their well-posedness.
    For example, Gonzales-Perez, Junge and Parcet \cite{GJP2017} studied the $L_p$-regularity of linear elliptic pseudodifferential equations on $\Rl^d_\theta$,
    and Chakraborty and Sinha have analysed diffusion equations \cite{Chakraborty-Sinha}. The related noncommutative torus $\Circ^d_\theta$ also received attention
    by those authors. The doctoral dissertation of Mingyu Zhao discussed Schauder and Strichartz estimates for $\Rl^d_\theta$ \cite{ZhaoThesis}, a topic
    which is relevant to the present paper. Concerning nonlinear equations on noncommutative spaces, progress has been made by Rosenberg \cite{Rosenberg2008}, who studied
    nonlinear partial differential equations on the noncommutative torus. The related topic of quantum field theory on noncommutative Euclidean spaces
    has received some interest, see e.g. \cite{DouglasNekrasov}.

    Besides \eqref{automatic_boundedness}, there are many counterintuitive features of noncommutative Euclidean spaces. Not all of these features play a role in the present article,
    but nonetheless they serve as an indication for why the exploration of partial differential equations on $\Rl^d_\theta$ may be interesting:
    \begin{enumerate}[{\rm (i)}]
        \item{} If $f$ and $g$ are two real valued functions on $\Rl^d$, we can be certain that their pointwise product $fg$ is real valued. By contrast, the product of two self-adjoint operators is self-adjoint if and only if they commute. This has implications for nonlinear operator-valued differential equations, where nonlinear terms might cause a self-adjoint initial condition to become non-self-adjoint.
        \item{} If $\det(\theta)\neq 0$, then the subspaces $\Sc(\Rl^d_\theta)$ and $C_0(\Rl^d_\theta)$ (defined below in Section \ref{preliminaries_section}) are ideals of $L_\infty(\Rl^d_\theta)$. That is, multiplication of a smooth ``function" by an arbitrary bounded ``function" yields a smooth function \cite{PSVZ2018}.
        \item{} If $\det(\theta)\neq 0$, then $L_\infty(\Rl^d_\theta)$ contains nontrivial infinitely differentiable idempotents. Specifically, the Schwartz space $\Sc(\Rl^d_\theta)$ is generated in a topological sense by idempotent self-adjoint elements \cite[Section 2.2]{gayral-moyal}. This is in contrast to $\Sc(\Rl^d)$, which contains only the zero idempotent.
        \item{} If $1 < p < \infty$, then the commutative $L_p$-space $L_p(\Rl^d)$ admits unconditional Schauder bases, for example orthonormal wavelet bases are unconditional in $L_p(\Rl^d)$ for $1 < p < \infty$. By contrast, when $\det(\theta)\neq 0$ and $p\neq 2$, the space $L_p(\Rl^d_\theta)$ admits no unconditional Schauder basis \cite[Theorem 5.1]{GordonLewis1974}.
        \item{} Classical Euclidean space $\Rl^d$ admits a dilation symmetry. That is, the multiplicative group $(0,\infty)$ acts by automorphisms on $L_\infty(\Rl^d)$ by dilation. There does not appear to be a natural notion of automorphisms by dilations on $L_\infty(\Rl^d_\theta)$. Instead, a dilation of the coordinate variables changes the value of $\theta$ \cite[Section 3.1]{MSX2}.
    \end{enumerate}
    
\subsection{Motivation: Noncommutative fluid mechanics}
    The original motivation for this paper arose from the specific example of the incompressible Navier-Stokes equations.
    The incompressible Navier-Stokes equations for $\Rl^d$ with viscosity $\nu > 0$ are conventionally stated in terms of an unknown velocity field $u:\Rl^d\times [0,T)\to \Rl^d$
    and pressure field $p:\Rl^d\times [0,T)\to \Rl$ depending on space-time variables $(x_1,\ldots,x_d,t)$ as follows:
    \begin{equation*}
        \partial_t u + u\cdot\nabla u = \nu\Delta u+\nabla p,\quad \nabla\cdot u = 0.
    \end{equation*} 
    Here, we write $\partial_t$ for differentiation in the time variable $t$, $\nabla$ denotes the gradient operator in the spatial variables, $\nabla = (\partial_{x_1},\ldots,\partial_{x_d})$,
    $\Delta u$ is the vector $(\Delta u_j)_{j=1}^d$ formed by applying the Laplace operator $\Delta = \partial_1^2+\cdots+\partial_d^2$ to each component of $u$, and $u\cdot \nabla u$ 
    denotes the nonlinear advective term
    \begin{equation}\label{advective_term}
        u\cdot \nabla u = (\sum_{k=1}^d u_k\partial_{x_k} u_j)_{j=1}^d.
    \end{equation}
    The local well-posedness theory of the incompressible viscous Navier-Stokes equations on $\Rl^d$ is a classical topic which is discussed in many references \cite{FujitaKato1964}, \cite[Chapter 5]{Bahouri-Chemin-Danchin-2011}, \cite[Section 17.4]{Taylor-pde-3-2011}, \cite[Chapter 7]{Lemarie-Riesset-Gilles-2016}. Conventionally, one eliminates the pressure term $\nabla p$
    by applying the Leray projection $\Pl$, defined as the $L_2$-orthogonal projection onto the closed subspace of divergence-free vector fields. The Leray projection annihilates $\nabla p$, and after some rearrangement and setting $\nu=1$ we are
    left with a single equation,
    \begin{equation}\label{leray_form}
        \partial_t u = -\Pl(\nabla\cdot (u\otimes u))+ \Delta u.
    \end{equation}
    where 
    $\nabla\cdot (u\otimes u)$ denotes the vector $(\sum_{j=1}^d \partial_j (u_ju_k))_{k=1}^d$.
    Formally, solutions to \eqref{leray_form} can be written in mild form
    \begin{equation}\label{mild_form}
        u(t) = e^{t\Delta}u(0)-\int_0^t e^{(t-s)\Delta}\Pl(\nabla\cdot(u(s)\otimes u(s)))\,ds,\quad 0\leq t < T
    \end{equation}
    where $s\mapsto e^{s\Delta}$ denotes the heat semigroup, and $u(s)$ is the value of the velocity field $u$ at time $s$. The classical local well-posedness theory for the Navier-Stokes equations is based on finding solutions to \eqref{mild_form} as a fixed-point problem.
    

    Well-established arguments (already known to Leray as early as 1934 \cite{Leray1934}) give the local well-posedness for \eqref{leray_form} when the initial data $u(0)$ is sufficiently regular (see also \cite{FujitaKato1964} or \cite[Theorem 5.2]{Bahouri-Chemin-Danchin-2011}, \cite[Chapter 17]{Taylor-pde-3-2011} and \cite[Chapter 7]{Lemarie-Riesset-Gilles-2016} for a modern exposition).
       
    The solution $u(t)$ can then be extended to a maximal time interval $t \in [0,T_{u(0)})$, and can be proved to satisfy the inequality
    \begin{equation}\label{energy_inequality}
        \|u(t)\|_{L_2(\Rl^d)\otimes \Rl^d} \leq \|u(0)\|_{L_2(\Rl^d)\otimes \Rl^d},\quad 0\leq t < T_{u_0}
    \end{equation}    
    whenever the initial data $u(0)$ is such that the right hand side is finite.
    The global well-posedness problem for given sufficiently regular initial data $u(0) \in L_2(\Rl^d)$ asks if $T_{u(0)} < \infty$ or $T_{u(0)} = \infty$. Leray
    established a {\it blow-up criterion}, which states that if the maximal time $T_{u(0)}$ is finite, then
    \begin{equation*}
        \sup_{0\leq t < T_{u(0)}} \|u(t)\|_{L_\infty(\Rl^d)\otimes \Rl^d} = \infty.
    \end{equation*}    
    Contrapositively, if $\|u(t)\|_{L_\infty(\Rl^d)\otimes \Rl^d}$ remains bounded for $t \in [0,T_{u(0)})$ then $T_{u(0)} = \infty$.
    Since the $L_2$ norm does not bound the $L_\infty$ norm, the $L_2$ norm inequality \eqref{energy_inequality} is not strong enough to prove global existence of solutions to \eqref{mild_form}. 
        
    The series of questions which motivated this paper are the following:
    \begin{enumerate}
        \item{} Can one formulate a suitable analogy of the incompressible Navier-Stokes equations in noncommutative Euclidean spaces?
        \item{} If so, can the local solution theory of these incompressible Navier-Stokes equations on $\Rl^d_\theta$ be developed in parallel to
                the classical case, and do we have the same result that solutions exist globally in time if their $L_\infty$ norm remains bounded?
        \item{} Finally, does \eqref{automatic_boundedness} imply that when $\det(\theta)\neq 0$ an analogy of the incompressible Navier-Stokes equations for $\Rl^d_\theta$ is globally well-posed in time?
    \end{enumerate}
    The answer to all three questions is affirmative, leading us to the potentially surprising result that one can prove global well-posedness for something analogous to the incompressible Navier-Stokes equations in the setting of $\Rl^d_\theta$ when $\det(\theta)\neq 0.$ Given the infamous difficulty of the corresponding problem on $\Rl^d$, this is a particularly striking
    example of the difference between commutative and noncommutative Euclidean spaces.
    
    Similar reasoning applies to some other evolution equations which conserve the $L_2$-norm, for example we can also consider analogies of the nonlinear Schr\"odinger equation, given
    in the commutative case as
    \[
        \ri\partial_t u = \Delta u + \mu u|u|^{p-1},\quad \mu\in \Rl.
    \]
    
%
    
\subsection{Motivation: The Meyer and L\"owner decompositions}
    A secondary motivation for this paper comes from a superficial similarity between the Meyer decomposition in paradifferential calculus and the L\"owner decomposition
    in the theory of operator-Lipschitz functions.
    Let $X$ and $Y$ be Banach spaces, with norms $\|\cdot\|_X$ and $\|\cdot\|_Y$ respectively. A function $F:X\to Y$ is
    said to be locally Lipschitz continuous if for every $R>0$ there exists a constant $C_R>0$ such that
    \[
        \|F(x_1)-F(x_2)\|_{Y} \leq C_R\|x_1-x_2\|_X,\quad \|x_1\|_X\leq R,\, \|x_2\|_X\leq R.
    \]
    That is, $F$ is Lipschitz continuous on bounded subsets of $X.$

    Let $F \in C^\infty(\Rl).$ In the classical setting of function spaces on $\Rl^d$, the nonlinear operation of applying a function composition
    \begin{equation*}
        u\mapsto F(u)
    \end{equation*}
    where $u$ belongs to some function space is sometimes called a Nemytskij operator \cite{Runst-Sickel-1996}. 
    It is an important problem to determine if a Nemytskij operator is locally Lipschitz between a given pair of function
    spaces on $\Rl^d.$ This is useful, for example, if one wishes to prove the existence of solutions to a nonlinear partial differential equation by fixed point iteration
    on some Banach function space.
    
    A well-known technique for studying Nemytskij operators is the so-called Meyer decomposition \cite[Chapter 13, Section 10]{Taylor-pde-3-2011}, \cite[Section 5.5.4]{Runst-Sickel-1996}. In a wide degree of generality, it is possible to
    find a pseudodifferential operator $m(F,u)$ and a smooth function $r(u)$ such that
    \begin{equation}\label{meyer_decomposition_intro}
        F(u) = m(F,u)u+r(u).
    \end{equation}
    Therefore, the function space mapping properties of the nonlinear operator $u\mapsto F(u)$ are reduced
    to the mapping properties of the linear pseudodifferential operator $m(F,u)$, up to a smooth error term. 
    The Meyer decomposition can also be applied to functions of several variables, $(u,v)\mapsto G(u,v).$ This is important in the study of differences of Nemytskij operators,
    as if $F$ is smooth then there exists a smooth function $G$ on $\Rl^2$ such that
    \begin{equation*}
        F(u)-F(v) = G(u,v)(u-v).
    \end{equation*}
    Applying the Meyer decomposition to $G$ and applying product estimates gives a method to prove local Lipschitz estimates for Nemytskij operators \cite[Chapter 2, Section 7]{Taylor-tools-for-pde-2000}.
    
    In a completely different setting, Lipschitz estimates are important in operator theory. If $A$ and $B$ are two bounded self-adjoint linear operators
    on a Hilbert space $H$, then a function $F$ on $\Rl$ is said to be operator Lipschitz if there exists a constant $C_F>0$ such that
    \[
        \|F(A)-F(B)\|_{\infty} \leq C_F\|A-B\|_{\infty}.
    \]
    Here, $\|\cdot\|_{\infty}$ is the operator norm and $F(A)$ and $F(B)$ are defined via functional calculus. 
    Not all Lipschitz functions are operator Lipschitz, even the absolute value function $F(t)=|t|$ is not operator Lipschitz \cite{Kato1973}.
    Similarly, one can also consider other Banach spaces of operators such as Schatten $\Lc_p$-spaces.
    
    One method to prove operator Lipschitz estimates comes from the theory of double operator integrals. This theory was originally invented by Daletskii and Krein \cite{Daletskii-Krein-1956}, and later extended by Birman and Solomyak \cite{Birman-Solomyak-I, Birman-Solomyak-II, Birman-Solomyak-III}. Some recent surveys on this topic are \cite{Aleksandrov-Peller-survey,Birman-Solomyak-2003,Peller-survey-2010,Peller-doi-2016}.
    The idea is to find a linear operator $\Ti^{A,B}_{F^{[1]}}$ on a space of operators such that
    \begin{equation}\label{lowner_formula_intro}
        F(A)-F(B) = \Ti^{A,B}_{F^{[1]}}(A-B).
    \end{equation}
    This is sometimes called a L\"owner identity, especially in the finite dimensional case.
    The intent is to study the properties of the nonlinear relationship between $A-B$ and $F(A)-F(B)$ by reducing them to the properties of the linear operator $\Ti^{A,B}_{F^{[1]}}$.
    
    Conceptually, the Meyer decomposition \eqref{meyer_decomposition_intro} and the L\"{o}wner identity \eqref{lowner_formula_intro} are similar since
    they both involve studying the mapping properties of a nonlinear operation on a Banach space by finding an equivalent linear operator. Despite these methods differing substantially in their technical details and domains of application, it is natural to ask if there is a mutual generalisation.
    
    Our goal is to combine the Meyer decomposition with double operator integration to develop a hybrid technique suitable for studying functional composition operators on noncommutative function spaces. This hybrid technique gives one possible shared generalisation of \eqref{meyer_decomposition_intro} and \eqref{lowner_formula_intro}.

\subsection{The difficulties with noncommutativity}
    We have already indicated that when $\det(\theta)\neq 0$, the ``measure theory" of $\Rl^d_\theta$ simplifies considerably.
    We wish to briefly explain why, despite this, there remain substantial mathematical difficulties in the analysis of $\Rl^d_\theta$. 
    
    The principal obstacle we have is that for $x,y \in L_\infty(\Rl^d_\theta)$ there is no bound
    \begin{equation*}
        |xy|^2 \leq \|y\|_\infty^2|x|^2.
    \end{equation*}
    Instead, one only has
    \begin{equation*}
        |yx|^2 \leq \|y\|_\infty^2|x|^2.
    \end{equation*}
    Here, $\leq$ is meant in the sense of operators, i.e. $A\leq B$ means that $\langle \xi,A\xi\rangle \leq \langle \xi,B\xi\rangle$ for all $\xi$ in the Hilbert space.
    This asymmetry between $|xy|$ and $|yx|$ is an obstacle to the study of the mapping properties of pseudodifferential operators
    on $L_p$-spaces for $\Rl^d_\theta$. The operators we study are generically of the form
    \begin{equation*}
        T(u) = \sum_{k=0}^\infty a_k\Delta_k(u)b_k,\quad u \in L_\infty(\Rl^d_\theta)
    \end{equation*}
    where $\{a_k\}_{k=0}^\infty$ and $\{b_k\}_{k=0}^\infty$ are sequences of elements of $L_\infty(\Rl^d_\theta)$, and $\{\Delta_k\}_{k=0}^\infty$
    is a Littlewood-Paley decomposition (to be defined below in Section \ref{lp_section}). Because there are no pointwise bounds $|xy|^2\leq \|y\|_\infty^2|x|^2$, even under restrictive conditions on $\{a_k\}$ and $\{b_k\}$ we have been unable to determine if operators of this form are bounded on $L_p$-Sobolev spaces. 
    Besov spaces are easier to treat in this regard, and this is why we work primarily with Besov spaces rather than Sobolev spaces.

\subsection{Main results}
    We summarise here the main results of this paper. Notation for Littlewood-Paley theory and Besov spaces will be introduced below in Section \ref{preliminaries_section}.
    
    The first result concerns boundedness on Besov spaces of a kind of pseudodifferential operator specified in terms of a Littlewood-Paley decomposition $\{\Delta_j\}_{j=0}^\infty$, to be defined below in Section \ref{lp_section}. This theorem will be proved in Section \ref{psdo_section}.
    \begin{theorem*}
        Define a linear operator of the form
        \begin{equation}\label{elementary_psdo_intro}
            T(u) = \sum_{j=0}^\infty a_j\Delta_j(u)b_j
        \end{equation}
        where $\{a_j\}$ and $\{b_j\}$ are sequences of smooth elements of $L_\infty(\Rl^d_\theta)$ such that for all multi-indices $\alpha\in \Ntrl^d$ we have
        \begin{equation*}
            \sup_{j\geq 0} 2^{-j|\alpha|}\|\partial^{\alpha} a_j\|_\infty < \infty,\quad \sup_{j\geq 0} 2^{-j|\alpha|}\|\partial^{\alpha}b_j\|_{\infty}.
        \end{equation*}
        The linear operator $T$ admits a bounded extension
        $$
            T:B^s_{p,q}(\Rl^d_\theta)\to B^s_{p,q}(\Rl^d_\theta)
        $$
        for all $s > 0$ and $p,q \in [1,\infty]$.
    \end{theorem*}
    This theorem is essentially a noncommutative version of the boundedness of pseudodifferential operators in the ``forbidden" symbol class $S^0_{1,1}$ on Besov spaces with a positive degree of smoothness. Compare \cite[Theorem 5.15]{Sawano2018}.
%
%
    
    The main purpose of the preceding theorem is to study non-linear composition operators 
    $$
    u\mapsto F(u),
    $$ 
    where $F \in C^\infty(\Rl)$ and $u\in L_\infty(\Rl^d_\theta)$ is self-adjoint. The main idea is to develop a kind of ``Meyer decomposition" based on the theory of double operator integration
    In brief, there exists a probability space $(\Omega,\mu)$ and functions $a_j:\Omega\to L_\infty(\Rl^d_\theta)$ and $b_j:\Omega\to L_\infty(\Rl^d_\theta)$ (depending on $u$)
    and a smooth remainder term $r(u)$ such that
    $$
        F(u) = r(u)+\int_{\Omega} \sum_{j=0}^\infty a_j(\omega)\Delta_j(u)b_j(\omega)\,d\mu(\omega).
    $$   
    We can treat the integrand as an operator of the form \eqref{elementary_psdo_intro} in order to apply the preceding theorem.
    
    In Section \ref{nemytskij_section} we obtain the following result that smooth functions on $\Rl$ are locally Lipschitz on $B^s_{p,q}(\Rl^d_\theta)\cap L_\infty(\Rl^d_\theta)$ for all $s > 0$.
    \begin{theorem*}
        Let $s > 0$ and $p,q\in [1,\infty]$ and let $u\in B^s_{p,q}(\Rl^d_\theta)\cap L_\infty(\Rl^d_\theta)$ be self-adjoint. If $F \in C^\infty(\Rl)$
        and $F(0)=0$, then
        $$
            F(u) \in B^s_{p,q}(\Rl^d_\theta).
        $$
    \end{theorem*}
    The assumption that $F$ be smooth can certainly be weakened, however for the moment the main goal is to demonstrate the technique.
    
    A related and somewhat simpler result is the following product estimate, established in Section \ref{multiplication_section}.
    \begin{theorem*}
        Let $s> 0$ and $p,q\in [1,\infty]$. Let $u,v \in B^s_{p,q}(\Rl^d_\theta)\cap L_\infty(\Rl^d_\theta)$. Then $uv \in B^s_{p,q}(\Rl^d_\theta)$, and
        $$
            \|uv\|_{B^s_{p,q}} \lesssim_{s,p,q} \|u\|_{L_\infty}\|v\|_{B^s_{p,q}}+\|u\|_{B^s_{p,q}}\|v\|_{L_\infty}.
        $$
    \end{theorem*}
    This result is proved by splitting the pointwise product $uv$ into three terms, in a manner similar to the well-known Bony decomposition. Similar results have appeared in unpublished work by G.~Hong. The corresponding commutative results are well known, see e.g. \cite[Theorem 4.36]{Sawano2018}, \cite[Chapter 4]{Runst-Sickel-1996}. 
    
    Combining the above nonlinear composition operator and product mapping estimates, we also obtain the following local Lipschitz estimate:
    \begin{theorem*}
        Let $s > 0$, $p,q \in [1,\infty]$ and $F \in C^\infty(\mathbb{R}).$
        If $u,v \in L_\infty(\Rl^d_\theta)\cap B^s_{p,q}(\Rl^d_\theta)$ are self-adjoint then there exists a constant $C_{F,\|u\|_\infty,\|u\|_{B^{s}_{p,q}},\|v\|_\infty,\|v\|_{B^s_{p,q}}}$
        such that
        $$
            \|F(u)-F(v)\|_{B^s_{p,q}} \leq C_{F,\|u\|_\infty,\|u\|_{B^{s}_{p,q}},\|v\|_\infty,\|v\|_{B^s_{p,q}}}\|u-v\|_{B^s_{p,q}}.
        $$
        In other words, the Nemytskij operator $u\mapsto F(u)$ is locally Lipschitz on the self-adjoint subspace of $B^s_{p,q}(\Rl^d_\theta)\cap L_{\infty}(\Rl^d_\theta).$
    \end{theorem*}
    
    Given local Lipschitz estimates for Nemytskij operators and some easily proved details about the heat semigroup $e^{t\Delta}$ on $L_p$-spaces, standard fixed point arguments render it possible to deduce the following well-posedness result:
    \begin{theorem*}
        Let $F \in C^\infty(\Rl)$ be real-valued, and let $u_0 \in B^s_{\infty,\infty}(\Rl^d_\theta)$ for some $s > 0$ be self-adjoint. Then there exists $T_{u_0} > 0$
        and a unique maximal solution
        \begin{equation*}
            u \in C([0,T_{u_0}),B^s_{\infty,\infty}(\Rl^d_\theta))\cap C^\infty((0,T_{u_0}),\cap_{r > 0} B^r_{\infty,\infty}(\Rl^d_\theta))
        \end{equation*}
        to the equation
        $$
            \frac{\partial u}{\partial t} = \Delta u+F(u),\quad 0 < t < T_{u_0}
        $$
        with initial condition $u_0$.
    \end{theorem*}
    
    These theorems so far represent noncommutative analogies of known theorems in the commutative case, covered in e.g. \cite{Bahouri-Chemin-Danchin-2011, Taylor-pde-3-2011,Runst-Sickel-1996}.
    The most interesting applications result from restricting attention to the strictly noncommutative $\det(\theta)\neq 0$ case. Recall
    that one form of nonlinear Schr\"odinger equation is stated as
    $$
        \ri \frac{\partial u}{\partial t} = \Delta u+\mu u|u|^{p-1}
    $$
    where $\mu\in \Rl$ and $p\geq 1$ \cite{GinibreVelo1979}. This equation is called \emph{focusing} or \emph{defocusing} depending on $\mu < 0$
    or $\mu > 0$ respectively. It is known that the focusing nonlinear Schr\"odinger equation is ill-posed \cite[Theorem 6.5.10]{Cazenave-semilinear-2003}.
    
    For the strictly noncommutative case of $\Rl^d_\theta$, we have the following:
    \begin{theorem*}
        Assume that $\det(\theta)\neq 0$. Let $\mu\in \Rl$ and let $p>1$ be an odd integer. For any $u_0 \in L_2(\Rl^d_\theta)$, there exists a unique 
        $$
            u \in C([0,\infty),L_2(\Rl^d_\theta))
        $$
        solving the nonlinear Schr\"odinger equation
        \begin{equation}\label{nonlinear_schrodinger_intro}
            \ri \frac{\partial u}{\partial t} = \Delta u + \mu u|u|^{p-1}
        \end{equation} 
        with initial condition $u_0$, in the mild sense.
    \end{theorem*}
    The proof is very straightforward. For similar reasons to the classical case, the $L_2$-norm is conserved by \eqref{nonlinear_schrodinger_intro}
    and due to \eqref{automatic_boundedness}, the $L_2$-norm is submultiplicative for $\det(\theta)\neq 0$.
    It follows almost immediately that local-in-time solutions can be extended indefinitely. The same proof applies for a slightly wider class of nonlinearities. 
    
    Finally, we can give a further example of the simplifications that occur when $\det(\theta)\neq 0$ through the study of the incompressible
    Navier-Stokes equations. The situation is similar to the nonlinear Schr\"odinger equation above, where the submultiplicativity
    of the $L_2$-norm implies that local-in-time solutions can be extended to global solutions. We provide more details below in Section \ref{fluids_section}.

\subsection{Outline of this paper}  
    This paper is organised as follows:
    \begin{itemize}
        \item{} In Section \ref{preliminaries_section} we recall facts about noncommutative Euclidean space.
        \item{} Section \ref{lp_section} discusses Littlewood-Paley theory for $\Rl^d_\theta$. To the best of our knowledge, this material is novel in the noncommutative setting although it is parallel to the classical case. In Subsection \ref{function_spaces_subsection} we define Sobolev and Besov spaces for $\Rl^d_\theta$. The definition of Sobolev spaces is standard, but the definition of Besov spaces is new.
        \item{} Section \ref{psdo_section} is concerned with ``elementary" pseudodifferential operators on $\Rl^d_\theta$. We prove operators of this form are bounded on Besov spaces of positive regularity.
        \item{} In Section \ref{multiplication_section}, we discuss the problem of multiplication on Besov spaces, based on a noncommutative version of the Bony decomposition.
                This material is novel, although similar in methods and results to the classical theory. We indicate the simplifications that arise when $\det(\theta)\neq 0$.
        \item{} It is in Section \ref{nemytskij_section} that we develop the main novelty of this paper. We discuss the theory of Nemytskij operators. That is,
                we study the operation of function composition $u\mapsto F(u)$ for self-adjoint $u\in L_\infty(\Rl^d_\theta)$, where $F\in C^\infty(\Rl)$. This is achieved with a noncommutative analogy of the Meyer decomposition which is based on the theory of double operator integrals. The most important result of this section implies that smooth functions $F \in C^\infty(\Rl)$ are locally Lipschitz on the self-adjoint subspace of $L_\infty(\Rl^d_\theta)\cap B^s_{p,q}(\Rl^d_\theta)$.
        \item{} Finally, in Section \ref{pde_section} we discuss nonlinear partial differential equations using the machinery developed in the preceding subsections.
    \end{itemize}
    Sections \ref{preliminaries_section} to \ref{nemytskij_section} remain agnostic about the value of $\theta$ and apply equally well in the commutative $\theta=0$ case, although
    the results are only novel when $\theta\neq 0$. In Section \ref{pde_section}, we explain how even though the preceding theory seemed to not depend on $\theta$, the theory of partial differential equations drastically simplifies when $\det(\theta)\neq 0$. 
    

\subsection{Acknowledgements}
    Thank you to  G.~Hong, M.~Junge, P.~Portal, H.~Sharma, X.~Xiong and D.~Zanin for helpful discussions.

\section{Preliminaries}\label{preliminaries_section}
    For a Hilbert space $H$, we denote by $\Bc(H)$ the algebra of all bounded linear endomorphisms of $H$. We denote by $L_p(\Rl^d)$ the $L_p$-spaces
    of pointwise almost-everywhere equivalence classes of $p$-integrable functions on Euclidean space $\Rl^d$.
    
    We use the notion of a weak$^*$ or Gel'fand integral \cite[pg. 53]{Diestel-Uhl}, in the setting described in \cite{DDSZ}. Given a $\sigma$-finite measure space $(X,\Sigma,\mu)$
    and a semifinite von Neumann algebra $(\Mv,\tau)$, a function $f:X\to \Mv$ is said to be weak$^*$-measurable if for all $x \in L_1(\Mv,\tau)$
    the map $\omega\mapsto \tau(xf(\omega))$ is measurable. A weak$^*$-integral $\int_X f\,d\mu$ is an element of $\Mv$ such that
    \begin{equation*}
        \tau\left(x\int_X f\,d\mu\right) = \int_X \tau(xf)\,d\mu,\quad x \in L_1(\Mv).
    \end{equation*} 
    If $f$ is weak$^*$-measurable, then the function $\omega\mapsto \|f(\omega)\|_{\Mv}$ is measurable\footnote{It is essential here that $\Mv$ admits a faithful representation on a \emph{separable} Hilbert space}, and if
    \begin{equation*}
        \int_{X} \|f\|_{\Mv} \,d\mu < \infty
    \end{equation*}
    then a unique weak$^*$-integral of $f$ exists.
    
    In the following subsections we introduce notation, terminology and basic results concerning $\Rl^d_\theta$. Most of the terminology is standard, and
    follows \cite{gayral-moyal}, \cite{GJP2017}, \cite[Section 6]{LeSZ-cwikel} and \cite{GJM}. Some related work is \cite{Werner1984,KeylKiukasWerner2016}.

\subsection{Definition of $\Rl^d_\theta$}\label{definition_subsection}
    As mentioned in the introduction, we will define $L_\infty(\Rl^d_\theta)$ as a von Neumann algebra generated by a unitary family $\{\lt(t)\}_{t\in \Rl^d}$
    satisfying the relation
    $$
        \lt(t)\lt(s) = \exp(\frac12\ri(t,\theta s))\lt(t+s),\quad t,s\in \Rl^d.
    $$
    While it is possible to define $L_\infty(\Rl^d_\theta)$ in an abstract operator-theoretic manner (this was the approach taken in \cite{GJP2017}), for simplicity we define the algebra as being generated by a concrete family of operators defined on the Hilbert space $L_2(\Rl^d)$.
    \begin{definition}
        For $t \in \Rl^d$, denote by $\lt(t)$ the operator on $L_2(\Rl^d)$ given by
        \begin{equation*}
            (\lt(t)\xi)(s) = \exp(\ri(t,s))\xi(s-\frac12\theta t),\quad \xi\in L_2(\Rl^d),\, t,s\in \Rl^d.
        \end{equation*}
        We define $L_\infty(\Rl^d_\theta)$ to be the weak operator topology closed subalgebra of $\Bc(L_2(\Rl^d))$ generated by the family $\{\lt(t)\}_{t\in \Rl^d}$.
    \end{definition}
    Observe that when $\theta=0$, the above definition reduces to the description of $L_\infty(\Rl^d)$ as the algebra of bounded pointwise multipliers on $L_2(\Rl^d)$.
    
    The structure of $L_\infty(\Rl^d_\theta)$ is determined by the Stone-von Neumann theorem \cite[Theorem 14.8]{Hall2013}, \cite[Theorem 5.2.2.2]{Bratteli-Robinson2}.
    \begin{theorem}\label{stone_von_neumann}
        The von Neumann algebra $L_\infty(\Rl^d_\theta)$ has type $\mathrm{I}$, and there is a canonical isomorphism
        \begin{equation*}
            \iota:L_\infty(\Rl^d_\theta) \to L_\infty(\Rl^{\dim(\ker(\theta))})\otimes \Bc(L_2(\Rl^{\frac{\rank(\theta)}{2}}))
        \end{equation*}
        where $\otimes$ is the von Neumann algebra tensor product.
        When $d=2$ and $\theta = \begin{pmatrix} 0 & -1 \\ 1 & 0\end{pmatrix}$, this isomorphism is given explicitly by
        \begin{equation*}
            \iota(\lt(t))\xi(r) = \exp(2\pi \ri t_2(r+\pi t_1))\xi(r+2\pi t_1),\quad r\in \Rl,\,\xi \in L_2(\Rl),\, t = (t_1,t_2)\in \Rl^2.
        \end{equation*}
    \end{theorem}
    
    \begin{remark}
        One way that this definition can be motivated is to introduce $\Rl^d_\theta$ as a ``space" with coordinates $\{x_1,\ldots,x_d\}$ obeying the commutation relation
        $$
            x_jx_k-x_kx_j = \ri \theta_{j,k}, \quad 1\leq j,k\leq d.
        $$
        For non-singular $\theta$, these relations define a Weyl algebra. We may then formally define $\lt(t)$ as
        \begin{equation*}
            \lt(t) = \exp(i (t_1x_1+t_2x_2+\cdots+t_dx_d)),\quad t \in \Rl^d.
        \end{equation*}
        A formal application of the Baker-Campbell-Hausdorff formula then leads to the relation
        $$
            \lt(t)\lt(s) = e^{\frac12 \ri (t,\theta s)}\lt(t+s),\quad t,s \in \Rl^d
        $$
        which is generally called the Weyl form of the canonical commutation relations.
    \end{remark}
    
\subsection{Measure theory and function spaces on $\Rl^d_\theta$}
    \begin{definition}
        Let $f \in L_1(\Rl^d)$. Define $\lt(f)\in L_\infty(\Rl^d_\theta)$ as the weak$^*$-integral
        \begin{equation*}
            \lt(f) = \int_{\Rl^d} f(t)\lt(t)\,dt.
        \end{equation*}
        The Weyl transform is defined as the composition of $\lt$ with the Fourier transform. We
        normalise the Fourier transform as
        \begin{equation*}
            \widehat{f}(\xi) = (2\pi)^{-\frac{d}{2}}\int_{\Rl^d} f(t)e^{-\ri (t,\xi)}\,dt,\quad \xi \in \Rl^d.
        \end{equation*}
        If $f$ has integrable Fourier transform, we define the Weyl transform $\wl(f)$ as
        \begin{equation*}
            \wl(f) := (2\pi)^{-\frac{d}{2}}\lt(\widehat{f}).
        \end{equation*}
    \end{definition}
    For $f\in L_1(\Rl^d),$ the integral defining $\lt(f)$ also converges in the $L_\infty(\Rl^d_\theta)$-valued Bochner sense. 
    \begin{remark}
        In terms of our heuristic description of $\Rl^d_\theta$ as a Euclidean space with noncommuting coordinates $\{x_1,\ldots,x_d\}$, we have
        \begin{equation*}
            \wl(f) = (2\pi)^{-\frac{d}{2}}\int_{\Rl^d} \widehat{f}(\xi_1,\ldots,\xi_d)e^{\ri(\xi_1x_1+\xi_2x_2+\cdots+\xi_dx_d)}\,d\xi.
        \end{equation*}
    \end{remark}
    
    The Schwartz space $\Sc(\Rl^d_\theta)$ is defined as the image of the classical Schwartz space under $\lt$. That is
    \begin{equation*}
        \Sc(\Rl^d_\theta) := \lt(\Sc(\Rl^d)).
    \end{equation*}
    Equivalently, $\Sc(\Rl^d_\theta) = \wl(\Sc(\Rl^d))$. We define a topology on $\Sc(\Rl^d_\theta)$ as the image of the canonical Fr\'echet topology on $\Sc(\Rl^d)$ under $\lt$.
    The topological dual of $\Sc(\Rl^d_\theta)$ is denoted $\Sc'(\Rl^d_\theta)$.
       
    Both the Weyl transform and $\lt$ are injective \cite[Subsection 2.2.3]{MSX2}. Given $f \in \Sc(\Rl^d)$, we define
    $$
        \tau(\lt(f)) := (2\pi)^{d}f(0).
    $$
    An important identity is that
    $$
        \tau(\lt(f)\lt(g)) = (2\pi)^{2d}f(0)g(0),\quad f,g \in \Sc(\Rl^d).
    $$
    Equivalently, we have
    \begin{equation*}
        \tau(\wl(f)) = \int_{\Rl^d} f(t)\,dt,\quad f \in \Sc(\Rl^d)
    \end{equation*}
    and
    \begin{equation}\label{unitary_weyl}
        \tau(\wl(f)\wl(g)) = \int_{\Rl^d} f(s)g(s)\,ds,\quad f,g \in \Sc(\Rl^d).
    \end{equation}
    Observe that $\wl(f)^* = \wl(\overline{f}),$ so that for all $f \in \Sc(\Rl^d)$ we have
    \begin{equation*}
        \tau(\lt(f)^*\lt(f)) = \tau(\wl(f)^*\wl(f)) = \int_{\Rl^d} |f(s)|^2\,ds
    \end{equation*}
    
    We can extend the Weyl transform to distributions. If $T \in \Sc'(\Rl^d)$, denote by $\wl(T)\in \Sc'(\Rl^d_\theta)$ the functional defined by
    \begin{equation*}
        (\wl(T),\wl(f)) = (T,f),\quad f \in \Sc(\Rl^d).
    \end{equation*}
    \begin{theorem}
        The functional $\tau:\Sc(\Rl^d_\theta)\to \Cplx$ uniquely extends to a normal semifinite trace on the von Neumann algebra $L_\infty(\Rl^d_\theta).$
        
        If $\theta=0$, then under the isomorphism \eqref{stone_von_neumann}, $\tau$ is exactly the Lebesgue integral. If $\det(\theta)\neq 0$, then $\tau$
        is (up to a normalisation) the operator trace on $\Bc(L_2(\Rl^{d/2}))$.
    \end{theorem}
    This theorem is proved in \cite{GJP2017} by constructing $L_\infty(\Rl^d_\theta)$ as an iterated cross product. It may also be proved by explicitly identifying 
    $\tau$ as the tensor product of the operator trace and the Lebesgue integral in terms of the isomorphism \eqref{stone_von_neumann}, see \cite[Proposition 2.4]{GJM} for details.
    If $\det(\theta)\neq 0$, then $\tau$ is related to the operator trace $\Tr$ by
    \begin{equation}\label{relation_to_classical_trace}
        \tau(u) = \det(2\pi \theta)^{\frac{1}{2}}\Tr(\iota(u)),\quad u \in \Sc(\Rl^d_\theta).
    \end{equation}
    where $\iota$ is the isomorphism \eqref{stone_von_neumann}. In particular, the range of the trace $\tau$ on projections consists of nonnegative integer multiples of
    $\det(2\pi \theta)^{\frac12}.$
    
    With this data, the pair $(L_\infty(\Rl^d_\theta),\tau)$ is a semifinite von Neumann algebra. As a special case of the theory of $L_p$-spaces corresponding
    to a semifinite von Neumann algebra, we have the following definition.
    \begin{definition}  
        Let $1\leq p < \infty.$ Define $L_p(\Rl^d_\theta)$ as the $L_p$-space associated to the semifinite trace $\tau$ on $L_\infty(\Rl^d_\theta)$. That is, let $N_p$
        denote the subspace of $x \in L_\infty(\Rl^d_\theta)$ such that
        $$
            \|x\|_p := \tau(|x|^p)^{1/p} < \infty.
        $$
        Then $L_p(\Rl^d_\theta)$ is defined as the completion of $N_p$ with respect to the norm $\|\cdot\|_p$.
        
        In particular, $L_2(\Rl^d_\theta)$ is the GNS Hilbert space of $L_\infty(\Rl^d_\theta)$ corresponding to the inner product
        $$
            \langle x,y\rangle := \tau(x^*y),\quad x,y \in N_2.
        $$
    \end{definition}
    The fact that $\|\cdot\|_p$ is a norm is a standard result in the theory of semifinite von Neumann algebras \cite[Theorem 4.4]{Fack-Kosaki}. There is a H\"older inequality
    for these $L_p$ spaces. That is, if $u \in L_p(\Rl^d_\theta)$ and $v \in L_q(\Rl^d_\theta),$ then $uv \in L_r(\Rl^d_\theta)$ where $\frac1r = \frac1p+\frac1q$
    and 
    \[
        \|uv\|_{r} \leq \|u\|_p\|v\|_q.
    \]   
    For $1\leq p < \infty$, the Schwartz space $\Sc(\Rl^d_\theta)$ is dense in $L_p(\Rl^d_\theta)$ \cite[Proposition 3.14]{MSX2}.
    Knowing this, it follows from \eqref{unitary_weyl} that $\wl$ extends to a unitary isomorphism
    $$
        \wl:L_2(\Rl^d)\to L_2(\Rl^d_\theta).
    $$
    The closure of $\Sc(\Rl^d_\theta)$ in the $L_\infty$ norm is denoted $C_0(\Rl^d_\theta)$. 
    
    At this point we again emphasise that when $\det(\theta)\neq 0$, Theorem \ref{stone_von_neumann} states that $L_\infty(\Rl^d_\theta)$ is isomorphic to the type $\mathrm{I}_\infty$ von Neumann
    algebra $\Bc(L_2(\Rl^{\frac{d}{2}}))$, and $\tau$ is proportional to the operator trace \cite[Section 6]{LeSZ-cwikel}. In this case, $L_p(\Rl^d_\theta)$ coincides with 
    the Schatten-von Neumann operator $\Lc_p$-space. It follows that if $\det(\theta)\neq 0$ then $L_p(\Rl^d_\theta)\subseteq L_q(\Rl^d_\theta)$ for all $p\leq q$.
    
    Still in the non-degenerate case, given an arbitrary $T \in \Sc'(\Rl^d),$ the operator $\wl(T)$ can be realized as a quadratic form on a dense subspace of $L_2(\Rl^{\frac{d}{2}}),$ for details see \cite[Section II.B]{Daubechies-weyl-quantization-1983}.
    
    \begin{remark}
        Daubechies has given a sufficient condition on a distribution $T$ such that (in our notation) $\wl(T) \in L_{\infty}(\Rl^d_\theta)$ \cite{Daubechies-bounded-operators-1980}. 
        It is also known that there exists $f \in L_{\infty}(\Rl^d)$ such that $\wl(f) \notin L_{\infty}(\Rl^d_\theta),$ and $f \notin L_1(\Rl^d)$ such that $\wl(f) \in L_1(\Rl^d_\theta)$
        \cite[Section III]{Daubechies-weyl-quantization-1983}.
    \end{remark}
    
\subsection{Differential calculus on $\Rl^d_\theta$}
    Differential calculus on $\Rl^d_\theta$ is based on the group of translations $\{T_s\}_{s \in \Rl^d}$, where $T_s$ is defined
    as the unique $*$-automorphism of $L_\infty(\Rl^d_\theta)$ which acts on $\lt(t)$ as
    \begin{equation*}
        T_s(\lt(t)) = \exp(\ri(t,s))\lt(t),\quad t,s \in \Rl^d.
    \end{equation*}
    Equivalently, for $x \in L_\infty(\Rl^d_\theta)\subseteq \Bc(L_2(\Rl^d))$ we may define $T_s(x)$ as the conjugation of $x$ by the unitary operator of translation by $s$ on $L_2(\Rl^d)$.
    \begin{definition}
    An element $x \in L_\infty(\Rl^d_\theta)+L_{1}(\Rl^d_\theta)$ is said to be smooth if for all $y\in L_1(\Rl^d_\theta)\cap L_\infty(\Rl^d_\theta)$ the function $s\mapsto \tau(yT_s(x))$ is smooth.
    \end{definition}
    The partial derivations $\partial_j$, $j=1,\ldots,d$ are defined on smooth elements $x$ by
    \begin{equation*}
        \partial_j x = \frac{d}{ds_j} T_s(x)|_{s=0}.
    \end{equation*}
    In terms of the map $\lt$ and the Weyl transform $\wl$, it is easily verified that
    \begin{equation*}
        \partial_j \lt(f) = \lt(\ri t_jf(t)),\quad \partial_j \wl(f) = \wl(\partial_j f),\; f \in\Sc(\Rl^d),\;j=1,\ldots,d.
    \end{equation*}
    For a multi-index $\alpha \in \Ntrl^d$, we define
    \begin{equation*}
        \partial^\alpha = \partial_1^{\alpha_1}\cdots\partial_d^{\alpha_d}.
    \end{equation*}    
    The Laplace operator $\Delta$ is defined as
    $$
        \Delta = \sum_{j=1}^d\partial_j^2.
    $$
    Equivalently, $\Delta$ is the image of the classical Laplace operator under the Weyl transform. That is,
    $$
        \Delta \wl(x)= \wl(\Delta x),\quad f \in \Sc(\Rl^d_\theta).
    $$
    We can extend $\partial^{\alpha}$ and $\Delta$ to distributions in the natural way. Namely if $T \in \Sc'(\Rl^d_\theta)$, 
    then we define
    $$
        (\partial^{\alpha} T,u) = (-1)^{|\alpha|}(T,\partial^{\alpha} u),\quad u \in \Sc(\Rl^d_\theta).
    $$   
    
\subsection{Convolution and Young's inequality}
    The following notion was used in \cite[Section 3.2]{MSX2}.
    \begin{definition}
        Let $1\leq p\leq\infty$ and $u \in L_p(\Rl^d_\theta).$
        
        For $K\in L_1(\Rl^d),$ we define
        \[
            K\ast u := \int_{\Rl^d} K(t)T_{-t}u\,dt
        \]
        as an $L_p(\Rl^d_\theta)$-valued Bochner integral when $p < \infty,$ and as a weak$^*$ integral when $p=\infty.$
    \end{definition}
    Since $L_p(\Rl^d_\theta)$ is separable when $p<\infty,$ the Bochner integrability of the integrand $K(t)T_{-t}u$ follows 
    from the integrability of $K$ and the $L_p$-norm continuity the function $t\mapsto T_{-t}u.$
    
    From the triangle inequality, we immediately obtain
    \begin{equation}\label{L_1_L_p_case}
        \|K\ast u\|_{L_{p}(\Rl^d_\theta)} \leq \|K\|_{L_1(\Rl^d)}\|u\|_{L_p(\Rl^d_\theta)},\quad K\in L_1(\Rl^d),\; u \in L_p(\Rl^d_\theta).
    \end{equation}
    
    A computation using the definition $T_s\lt(t) = e^{i (s,t)}\lt(t)$ shows that
    \[
        K\ast \lt(f) = (2\pi)^{\frac{d}{2}}\lt(\hat{K}f),\quad f, K \in L_1(\Rl^d).
    \]
    The inequality $\|\lt(f)\|_{L_{\infty}(\Rl^d_\theta)} \leq \|f\|_{L_1(\Rl^d)}$ implies that
    \[
        \|K\ast \lt(f)\|_{L_{\infty}(\Rl^d_\theta)} \leq (2\pi)^{\frac{d}{2}}\|\hat{K}f\|_{L_1(\Rl^d)}.
    \]
    By the Cauchy-Schwartz inequality and $\|\lt(f)\|_2=(2\pi)^{\frac{d}{2}}\|f\|_2,$ it follows that
    \[
        \|K\ast \lt(f)\|_{L_\infty(\Rl^d_\theta)} \leq \|K\|_{L_2(\Rl^d)}\|\lt(f)\|_{L_2(\Rl^d_\theta)}.
    \]
    It follows that the bilinear operator $(K,u)\mapsto K\ast u$ admits a bounded extension to a mapping from $L_2(\Rl^d)\times L_2(\Rl^d_\theta)$ into $L_\infty(\Rl^d_\theta),$
    with norm bound
    \begin{equation}\label{L_2_L_2_case}
        \|K\ast u\|_{L_{\infty}(\Rl^d_\theta)} \leq \|K\|_{L_2(\Rl^d)}\|u\|_{L_2(\Rl^d_\theta)},\quad K \in L_2(\Rl^d),\; u \in L_2(\Rl^d_\theta).
    \end{equation}

    The following is a noncommutative substitute for Young's convolution inequality. The simple proof below only works for $p\geq 2,$ for the general case see the argument in \cite[Section 3.1]{Lafleche2022}.
    \begin{theorem}\label{young_inequality}
        Let $1\leq p,q,r \leq \infty$ obey the relation
        \[
            \frac{1}{r}+1 = \frac{1}{p}+\frac{1}{q}.
        \]
        Then the bilinear map $(K,u)\mapsto K\ast u$
        admits a continuous extension to a mapping from $L_p(\Rl^d)\times L_q(\Rl^d_\theta)$ into $L_r(\Rl^d_\theta)$, and we have the inequality
        \[
            \|K\ast u\|_{L_r(\Rl^d_\theta)} \leq \|K\|_{L_p(\Rl^d)}\|u\|_{L_q(\Rl^d_\theta)},\quad K \in L_p(\Rl^d),\; u \in L_q(\Rl^d_\theta).
        \]
    \end{theorem}
    \begin{proof}
        Initially suppose that $1\leq p\leq 2.$ We obtain the inequality from bilinear interpolation applied to \eqref{L_1_L_p_case} and \eqref{L_2_L_2_case}.
        
        Let $\Sigma$ denote the set of $(\frac1r,\frac1p,\frac1q) \in [0,1]^3$ such that 
        \begin{equation}\label{youngs_inequality_goal}
            \|K\ast u\|_{L_r(\Rl^d_\theta)} \leq \|K\|_{L_p(\Rl^d)}\|u\|_{L_q(\Rl^d_\theta)}
        \end{equation}
        The results \eqref{L_1_L_p_case} with $p=1$ and $p=\infty$ and \eqref{L_2_L_2_case} verify that
        \[
            (1,1,1),(0,1,0),(0,\frac12,\frac12) \in \Delta.
        \]        
        Since the noncommutative $L_p$-spaces are closed under complex interpolation, multilinear interpolation (c.f. \cite[Theorem 4.4.2]{Bergh-Lofstrom-1976}) implies that $\Delta$
        is convex.
        
        Denote $\widetilde{K}(t) := K(-t).$ Given $u,v \in \Sc(\Rl^d_\theta),$ it is readily verified that
        \[
            \tau(v(K\ast u) = \tau((\widetilde{K}\ast v)u).
        \]
        Thus, if $(\frac1r,\frac1p,\frac1q)\in \Delta,$ then H\"older's inequality implies that
        \[
            |\tau(v(K\ast u))| \leq \|u\|_{r'}\|\widetilde{K}\ast v\|_{r'}\leq \|u\|_{r'}\|K\|_{p}\|v\|_q.
        \]
        Here, $r' = \frac{r}{r-1}.$ Denoting $q' = \frac{q}{q-1},$ it follows that
        \[
            \|K\ast u\|_{q'} \leq \|K\|_p\|u\|_{r'}.
        \]
        Thus $(1-\frac1q,\frac1p,1-\frac1r) \in \Delta.$
        
        Since $(0,\frac12,\frac12)\in \Delta,$ it follows that $(\frac12,\frac12,1)\in \Delta.$ By convexity, we have
        \[
            (\frac1r,\frac1p,\frac1q) \in \Delta
        \]
        for all $1+\frac1r=\frac1p+\frac1q,$ with $p\geq 2.$
        
    \end{proof}

\section{Multipliers and Littlewood-Paley theory for $\Rl^d_\theta$}\label{lp_section}
    The Weyl transform defines a linear topological isomorphism
    $$
        \wl:\Sc'(\Rl^d)\to \Sc'(\Rl^d_\theta)
    $$
    which intertwines the partial derivative operators $\{\partial^{\alpha}\}_{\alpha\in \Ntrl^d}$ on $\Rl^d$
    with those on $\Rl^d_\theta$,
    \begin{equation*}
        \wl(\partial^{\alpha}T) = \partial^\alpha\wl(T),\quad T \in \Sc'(\Rl^d).
    \end{equation*}
    For the sake of brevity, let $D_j$ denote the rescaled partial derivation
    $$
        D_j = \frac{1}{\ri}\partial_j,\quad j=1,\ldots,d.
    $$
    For a multi-index $\alpha = (\alpha_1,\ldots,\alpha_d)\in \Ntrl^d$, we denote $D^{\alpha}$ for $D_1^{\alpha_1}\cdots D_d^{\alpha_d}$.
    \begin{definition}
        If $m \in C^\infty(\Rl^d)$ has at most polynomial increase at infinity, with all derivatives having at most polynomial growth, let $m(D)$ denote the corresponding Fourier multiplier. That is, for $f \in \Sc(\Rl^d)$, define
        $$
            (m(D)f)(t) = (2\pi)^{-\frac{d}{2}}\int_{\Rl^d} m(\xi)\widehat{f}(\xi)e^{\ri (t,\xi)}\,d\xi,\quad t \in \Rl^d.
        $$
        
        We define $m(D)$ on $\Sc(\Rl^d_\theta)$ using the Weyl transform. That is,
        $$
            m(D)\wl(u) := \wl(m(D)u),\quad u \in \Sc(\Rl^d_\theta).
        $$
        Equivalently, if $\wl(f) \in \Sc(\Rl^d_\theta)$, then
        $$
            m(D)\wl(f) = \int_{\Rl^d} m(\xi)\widehat{f}(\xi)\lt(\xi)\,d\xi = \wl(m(D)f).
        $$
        The definition of $m(D)$ is extended to $\Sc'(\Rl^d_\theta)$ by the relation
        $$
            (m(D)T,u) = (T,m(D)u),\quad u \in \Sc(\Rl^d_\theta),\,T \in \Sc'(\Rl^d_\theta).
        $$
    \end{definition}
    Denote by $\check{m}$ the inverse Fourier transform of $m$,
    \begin{equation*}
        \check{m}(t) = (2\pi)^{-\frac{d}{2}}\int_{\Rl^d} m(\xi)e^{\ri (\xi,t)}\,d\xi,\quad t \in \Rl^d.
    \end{equation*}
    Observe that if $\check{m} \in L_1(\Rl^d)$, then
    \begin{equation*}
        m(D)u = (2\pi)^{-d}\int_{\Rl^d} \check{m}(\xi)T_{-\xi}(u)\,d\xi,\quad u \in \Sc(\Rl^d_\theta).
    \end{equation*}
    That is,
    \[
        m(D)u = (2\pi)^{-d}\check{m}\ast u,\quad u \in L_p(\Rl^d_\theta),\;\check{m}\in L_1(\mathbb{R}^d).
    \]
    
    It follows immediately from Theorem \ref{young_inequality} that we have the following:
    \begin{theorem}\label{weak_young_inequality}
        Assume $m$ is such that $\check{m} \in L_1(\Rl^d)$. For all $1\leq p < \infty$, $m(D)$ extends to a bounded linear map from $L_p(\Rl^d_\theta)$ to $L_p(\Rl^d_\theta)$, with norm
        $$
            \|m(D)\|_{L_p(\Rl^d_\theta)\to L_p(\Rl^d_\theta)} \leq \|\check{m}\|_{L_1(\Rl^d)}.
        $$
        Moreover, $m(D)$ is bounded in the $L_\infty$ norm, and extends by duality to a linear map
        $$
            \|m(D)\|_{L_\infty(\Rl^d_\theta)\to L_\infty(\Rl^d_\theta)} \leq \|\check{m}\|_{L_1(\Rl^d)}.
        $$
        More generally, for any $1\leq p\leq q \leq\infty,$ we have
        $$
            \|m(D)\|_{L_p(\Rl^d_\theta)\to L_q(\Rl^d_\theta)} \leq \|\check{m}\|_{L_{r}(\Rl^d)}
        $$        
        where $\frac{1}{r} = 1+\frac{1}{q}-\frac{1}{p}.$
    \end{theorem}
        
    A simple application of Theorem \ref{weak_young_inequality} is the following Bernstein-type estimate.    
    \begin{corollary}\label{bernstein_inequality}
        Let $f \in L_1(\Rl^d)$ be supported in a ball of radius $\sigma>0$. There exists a constant $C_d$ such that for all $1\leq p\leq q\leq \infty$ we have
        \[
            \|\lt(f)\|_{L_q(\Rl^d)} \leq C_d\sigma^{d\left(\frac{1}{p}-\frac{1}{q}\right)}\|\lt(f)\|_{L_p(\Rl^d_\theta)}.
        \]
    \end{corollary}
    \begin{proof}
        By translation invariance it suffices to assume that $f$ is supported in a ball of radius $\sigma$ centred at the origin.
        Let $\varphi$ be a smooth function on $\Rl^d$ which is identically equal to $1$ on the unit ball of radius $1$
        and vanishes outside a ball of radius $2.$ If we denote
        \[
            \varphi_{\sigma}(t) = \varphi(\sigma^{-1}t)
        \]
        then
        \[
            \lt(f) = \lt(\varphi_{\sigma}f) = \check{\varphi_{\sigma}}\ast \lt(f).
        \]  
        For $1\leq r\leq \infty,$ we have
        \[
            \|\check{\varphi}_{\sigma}\|_{L_r(\Rl^d)} = \sigma^{d\left(1-\frac{1}{r}\right)}\|\check{\varphi}\|_{L_r(\Rl^d)}.
        \]
        Since $p\leq q,$ then there exists $r\geq 1$ such that $\frac{1}{r} = 1+\frac{1}{q}-\frac{1}{p}$ and applying Theorem \ref{weak_young_inequality} yields
        \[
            \|\lt(f)\|_{L_p(\Rl^d_\theta)} \leq \sigma^{d\left(\frac{1}{p}-\frac{1}{q}\right)}\|\check{\varphi}\|_{L_r(\Rl^d)}\|\lt(f)\|_{L_q(\Rl^d_\theta)}.
        \]
        Taking $C_d = \|\check{\varphi}\|_{L_r(\Rl^d)}$ completes the proof.
    \end{proof}
    
    \begin{remark}
        It should be noted that in the commutative case, it suffices to prove Corollary \ref{bernstein_inequality} with $\sigma=1$, 
        since the constant $\sigma^{d(\frac{1}{p}-\frac{1}{q})}$ can be recovered by a rescaling. In the noncommutative case this is not possible
        because as mentioned in the introduction, there is no appropriate ``dilation" action on $L_{\infty}(\Rl^d_\theta).$ 
        
        We also point out that while Corollary \ref{bernstein_inequality} is valid for all $\theta,$ the statement is trivial when $\det(\theta)\neq 0$.
        In that case, for $1\leq p \leq q\leq \infty$ it follows from \eqref{relation_to_classical_trace} that
        \[
            \det(2\pi\theta)^{\frac{1}{2p}-\frac{1}{2q}}\|u\|_{L_q(\Rl^d_\theta)} \leq \|u\|_{L_p(\Rl^d_\theta)},\quad u\in L_p(\Rl^d_\theta).
        \]
        This is far stronger than Corollary \ref{bernstein_inequality}.
    \end{remark}

%
    
    We now construct the homogeneous $\{\dot{\Delta}_j\}_{j\in \Itgr}$ and inhomogeneous $\{\Delta_j\}_{j=0}^\infty$ Littlewood-Paley decompositions. This development is directly
    in line with the commutative case, as in e.g. \cite[Chapter 5]{Grafakos-1}. The essential difference is that we use the Weyl transform $\wl$ in place of the Fourier transform.
    Denote by $B(0,r)$ the ball in $\Rl^d$ of radius $r$ centered at zero. 
    Let $\Psi$ be a smooth radial function on $\Rl^d$ such that
    \begin{equation*}
        \mathrm{supp}(\Psi) \subseteq B(0,2)\setminus B(0,1/2)
    \end{equation*}
    and chosen such that
    \begin{equation*}
        \Psi(\xi)+\Psi\left(\frac{\xi}{2}\right) = 1,\quad \xi \in B(0,2)\setminus B(0,1).
    \end{equation*}
    Let $\Phi$ be a smooth radial function supported in $B(0,1)$ such that
    \begin{equation*}
        \Phi(\xi) + \Psi(\xi) = 1,\quad \xi \in B(0,1).
    \end{equation*}
    (that is, $\Phi = 1-\Psi$ on $B(0,1)$ and zero elsewhere).
    
    For $j\in \Itgr$, let $\Psi_j(\xi) := \Psi(2^{-j}\xi)$. Then
    \begin{equation*}
        \Phi+\sum_{k=0}^\infty \Psi_k = 1
    \end{equation*}
    while for $\xi\neq 0$ we have $\sum_{j\in\Itgr} \Psi_j(\xi)=1.$
    Define the operator $\dot{\Delta}_j$ on $\Sc'(\Rl^d_\theta)$ by
    $$
        \dot{\Delta}_j = \Psi_j(D),\quad j \in \Itgr.
    $$
    Specifically, for $f \in \Sc(\Rl^d_\theta)$ we have,
    \begin{equation*}
        \dot{\Delta}_j\lt(f) = \lt(\Psi_j f),\quad j\in \Itgr.
    \end{equation*}
    Define
    \begin{equation*}
        R = \Phi(D).
    \end{equation*}
    For $j\geq 0$, define
    \begin{equation*}
        S_n = R+\sum_{j=0}^n \dot{\Delta}_j.
    \end{equation*}
    The inhomogeneous Littlewood-Paley decomposition $\{\Delta_j\}_{j=0}^\infty$ is defined as
    \begin{equation*}
        \Delta_j = \begin{cases} 
                        \dot{\Delta}_j,\quad j \geq 1,\\
                        R+\dot{\Delta}_0,\quad j=0.
                    \end{cases}
    \end{equation*}
    With this notation, we have $S_k = \sum_{j=0}^k \Delta_j$ and $\Delta_j = S_{j}-S_{j-1}$ for $j\geq 1$.
    
    An immediate consequence of Corollary \ref{bernstein_inequality} is that for all $1\leq p\leq q\leq\infty$ we have
    \begin{equation}\label{bernstein_for_littlewood_paley}
        \|\Delta_j u\|_{L_q(\Rl^d_\theta)} \leq C_d2^{jd\left(\frac{1}{p}-\frac{1}{q}\right)}\|\Delta_j u\|_{L_p(\Rl^d_\theta)}
    \end{equation}
    
    
    The following fact is purely commutative, and well-known. For a proof see e.g. \cite[Lemma 8.2.4]{ArendtBattyHieberNeubrander2011}.
    \begin{lemma}\label{mikhlin_theorem}
        Suppose that $m \in C^\infty(\Rl^d\setminus \{0\})$ is a function such that for all $\alpha$ we have
        \begin{equation*}
            |\partial_\xi^{\alpha}m(\xi)|\leq C_{\alpha,m}|\xi|^{s-|\alpha|},\quad \xi \in \Rl^d\setminus \{0\}
        \end{equation*}
        where $s \in \Rl$ is fixed. Denote by $m_j$ the function
        \begin{equation*}
            m_j(\xi) = m(\xi)\Psi(2^{-j}\xi).
        \end{equation*}
        Then $\check{m}_j$ is integrable, with norm bounded by
        \begin{equation*}
            \|\check{m}_j\|_{L_1(\Rl^d)} \lesssim_{m} 2^{sj},\quad j\in \Itgr.
        \end{equation*}
        (That is, the implicit constant depends on $m$ but is uniform in $j$.)
    \end{lemma}

    Observe that if $m$ is merely a smooth function on $\Rl^d$, then the inverse Fourier transform of $m\Phi$ is Schwartz class, and in particular is integrable.
    
    A combination of Theorem \ref{weak_young_inequality} and Lemma \ref{mikhlin_theorem} yields the following:
    \begin{corollary}\label{nc_mikhlin}
        Let $m$ satisfy the conditions of Lemma \ref{mikhlin_theorem}. Then for all $1\leq p \leq \infty$ we have
        \begin{equation*}
            \|m(D)\dot{\Delta}_j\|_{L_p(\Rl^d_\theta)\to L_p(\Rl^d_\theta)} \lesssim_{d,m} 2^{sj},\quad j \in \Itgr.
        \end{equation*}
        If $m$ is in addition smooth at the origin, then
        \begin{equation*}
            \|m(D)\Delta_j\|_{L_p(\Rl^d_{\theta})\to L_p(\Rl^d_\theta)} \lesssim_{d,m} 2^{sj},\quad j\geq 0.
        \end{equation*}
    \end{corollary}

    \begin{remark}
        There is also a Mikhlin multiplier theorem for $\Rl^d_\theta.$ It was proved in \cite{MSX3} that if $m\in L_{\infty}(\Rl^d)$ is a Mihklin multiplier, i.e.
        if $\sup_{|\xi|\neq 0} |\xi|^{\alpha}|\partial_{\xi}^{\alpha}m(\xi)| < \infty$ for $|\alpha|\leq \lfloor \frac{d}{2}\rfloor+1,$
        then $m(D)$ is bounded on $L_p(\Rl^d_\theta)$ for $1<p<\infty.$ This subsumes Corollary \ref{nc_mikhlin} when $1<p<\infty.$
    \end{remark}
       
    We record some simple consequences of Corollary \ref{nc_mikhlin}.
    \begin{proposition}\label{homogeneous_bound}
        Let $1\leq p \leq \infty$ and $j\in \Itgr$. If $u \in \Sc'(\Rl^d_\theta)$, then
        \begin{enumerate}[{\rm (i)}]
            \item{} Let $m$ be a function satisfying the conditions of Corollary \ref{nc_mikhlin}, and which is smooth at $0$. Then for $j\geq 0$ we have
            \begin{equation*}
                \|m(D)\Delta_j u\|_p \lesssim_m 2^{js}\|\Delta_j u\|_p.
            \end{equation*}
            \item{} For all multi-indices $\alpha\in \Ntrl^d$, we have
            \begin{equation*}
                \|D^\alpha \Delta_ju\|_p \lesssim_{\alpha} 2^{j|\alpha|}\|\Delta_j u\|_p.
            \end{equation*}
            \item{}\label{fourier_truncation_bound} For all $\alpha\in \Ntrl^d$, we have the bound
            \begin{equation*}
                \|D^\alpha S_ju\|_p \lesssim 2^{j|\alpha|}\|u\|_p,\quad j\geq 0.
            \end{equation*}
        \end{enumerate}
    \end{proposition}
    
    The following proposition is not hard to prove (these results other than \eqref{equivalence_of_B_22} also be extracted from \cite[Section 3.2]{MSX2}).
    \begin{proposition}\label{approximation_facts}
        Let $u \in \Sc'(\Rl^d_\theta)$. Then
        \begin{enumerate}[{\rm (i)}]
            \item{} For all $j\geq 0$ and all $p \in [1,\infty]$ we have
                \begin{equation*}
                    \|\Delta_j u\|_p\leq \|u\|_p,\quad \|S_ju\|_p\leq \|u\|_p.
                \end{equation*}
            \item{} For all $p\in [1,\infty)$ if $u \in L_p(\Rl^d_\theta)$, we have
            \begin{equation*}
                \lim_{j\to \infty} \|S_ju-u\|_p = 0.
            \end{equation*}
            If $p =\infty$, then the same holds provided that $u \in C_0(\Rl^d_\theta)$.
            \item{}\label{equivalence_of_B_22} There is an equivalence of norms
            \begin{equation*}
                \|u\|_{2} \cong \left(\sum_{j=0}^\infty \|\Delta_j u\|_{2}^2\right)^{1/2}.
            \end{equation*}
        \end{enumerate}
    \end{proposition}

\subsection{Heat and Schr\"odinger semigroups on $\Rl^d_\theta$}
    We will concern ourselves with two semigroups of operators on $L_2(\Rl^d_\theta)$, the heat and Schr\"odinger semigroups
    \begin{equation*}
        t\mapsto e^{t\Delta},\quad t\mapsto e^{\ri t\Delta}.
    \end{equation*}
    These operators can be defined by functional calculus on $L_2(\Rl^d_\theta)$, or equivalently as Fourier multipliers. Given $T \in \Sc'(\Rl^d)$, we may define
    \begin{equation*}
        e^{t\Delta}\lt(T) = \lt(e^{-|\xi|^2t}\cdot T),\quad e^{\ri t\Delta}\lt(T) = \lt(e^{-\ri |\xi|^2t}\cdot T).
    \end{equation*}
    Equivalently, the semigroups $e^{t\Delta}$ and $e^{\ri t\Delta}$ are the images of the classical heat and Schr\"odinger semigroups
    on $\Rl^d$ under the Weyl transform $\wl$.
    
    Recently, a theory of elliptic pseudodifferential operators on $\Rl^d_\theta$ has been developed \cite{GJP2017}. It is plausible that many of the following results
    generalise to semigroups generated by accretive elliptic differential operators, although at present such results are not known. 
    
    \begin{theorem}\label{heat_mapping}
        For all $1\leq p \leq \infty$ and $t > 0$, the operator $e^{t\Delta}$ is bounded on $L_p(\Rl^d_\theta)$, with norm
        $$
            \|e^{t\Delta}\|_{L_p(\Rl^d_\theta)\to L_p(\Rl^d_\theta)} \leq 1.
        $$
        If $p < \infty$, then $t\mapsto e^{t\Delta}$ is strongly continuous on $L_p(\Rl^d_\theta)$, in the sense that the mapping
        $$
            [0,\infty)\times L_p(\Rl^d_\theta)\to L_p(\Rl^d_\theta),\quad (t,u) \mapsto e^{t\Delta} u
        $$
        is continuous.
    \end{theorem}
    \begin{proof}
        Observe that by definition, $e^{t\Delta}$ is an operator of the form $m(D)$, where
        $$
            m(\xi) = e^{-t|\xi|^2},\quad \xi \in \Rl^d.
        $$
        and the inverse Fourier transform is given by the classical formula
        $$
            \check{m}(s) = (4\pi t)^{-\frac{d}{2}}e^{-\frac{|s|^2}{4t}},\quad s \in \Rl^d.
        $$
        Since $\|\check{m}\|_1 = 1$, the $L_p(\Rl^d_\theta)\to L_p(\Rl^d_\theta)$ norm bound follows from Theorem \ref{weak_young_inequality}.
                
        The strong continuity amounts to proving that
        $$
            \lim_{t\to 0} \|u-e^{t\Delta}u\|_p = 0,\quad u\in L_p(\Rl^d_\theta)\text{ for }p<\infty
        $$
        or $u\in C_0(\Rl^d_\theta)$ for $p=\infty.$ This is established as a special case of \cite[Theorem 3.10]{MSX2}.
    \end{proof}
    
    Immediately from the unitarity of the Weyl transform, we have the following:
    \begin{theorem}\label{schr_mapping}
        On $L_2(\Rl^d_\theta)$, $t\mapsto e^{\ri t\Delta}$ is a strongly continuous unitary group.
    \end{theorem}
    Observe that a combination of Proposition \ref{approximation_facts}, Theorem \ref{heat_mapping} and Theorem \ref{schr_mapping} yields
    \begin{align*}
            \|e^{t\Delta}\Delta_j\|_{L_p(\Rl^d_\theta)\to L_p(\Rl^d_\theta)} &\lesssim 1,\\
        \|e^{\ri t\Delta}\Delta_j\|_{L_2(\Rl^d_\theta)\to L_2(\Rl^d_\theta)} &\lesssim 1
    \end{align*}
    for all $p \in [1,\infty]$, where the constants are uniform in $j\geq 0$.
    
\subsection{Function spaces on $\Rl^d_\theta$}\label{function_spaces_subsection}
    Let $J := (1-\Delta)^{1/2} = (1+D^2)^{1/2}$. This is the Bessel potential operator. 
    
    The definition of $L_p$-Sobolev spaces for $\Rl^d_\theta$ is now standard. 
    We will have little use for $L_p$-Sobolev spaces for $p\neq 2$, but we record the definition here for the sake of completeness. The following is identical to \cite[Section 3.2]{MSX2} and \cite{GJP2017}.   \begin{definition}
        For $1\leq p < \infty$ and $s \in \Rl$, define the Bessel potential Sobolev space $W^s_p(\Rl^d_\theta)$ as the subset of $u \in \Sc'(\Rl^d_\theta)$ such that $J^su \in L_p(\Rl^d_\theta)$, with the norm
        \begin{equation*}
            \|u\|_{W^s_p} = \|J^su\|_{p}.
        \end{equation*}     %
    \end{definition}

    Besov spaces for $\Rl^d_\theta$ can be defined by the Littlewood-Paley decomposition.
    \begin{definition}\label{besov_spaces_definition}
%
        Let $s\in \Rl$ and $p,q \in [1,\infty]$. The inhomogeneous Besov class $B^s_{p,q}(\Rl^d_\theta)$ is defined as the subspace of
        distributions $u\in \Sc'(\Rl^d_\theta)$ such that
        \begin{equation*}
            \|u\|_{B^s_{p,q}} := \left(\sum_{j= 0}^\infty 2^{jsq}\|\Delta_j u\|_p^q\right)^{1/q} < \infty
        \end{equation*}
        for $q < \infty$, and
        \begin{equation*}
            \|u\|_{B^s_{p,\infty}} = \sup_{j\geq 0} 2^{sj}\|\Delta_j u\|_p.
        \end{equation*}
    \end{definition}
    It is not hard to show that the space $B^s_{p,q}(\Rl^d_\theta)$ does not depend on the choice of function $\Psi$ defining the Littlewood-Paley decomposition,
    up to equivalence of norms. This can be proved by an identical method to \cite[Remark 3.2]{XXY2018}.
    
    \begin{remark}
        Very recently, L.~Lafleche \cite{Lafleche2022} has introduced a family of quantum Besov spaces related to Weyl quantization. I conjecture that these are essentially the same as Definition \ref{besov_spaces_definition}.
    \end{remark}


    The following lemma lists some straightforward consequences of the definition of $B^s_{p,q}(\Rl^d_\theta).$
    \begin{lemma}
        The Besov spaces have the following elementary properties:
        \begin{enumerate}[{\rm (i)}]
            \item{}\label{sobolev_inclusion} If $r>s,$ then $\|u\|_{B^s_{p,q}(\Rl^d_\theta)} \lesssim_{s,r} \|u\|_{W^r_p(\Rl^d_\theta)}.$
            \item{}\label{schwartz_dense} If $s\in \Rl,$ $p\in [1,\infty)$ and $q < \infty,$ the space $\Sc(\Rl^d_\theta)$ is dense in $B^s_{p,q}(\Rl^d_\theta)$. 
            \item{} For every $s >0$ and $p,q \in [1,\infty]$, we have $B^s_{p,q}(\Rl^d_\theta)\subseteq L_p(\Rl^d_\theta).$
            \item{} For every $p\in  [1,\infty]$, we have $L_p(\Rl^d_\theta)\subseteq B^0_{p,\infty}(\Rl^d_\theta)$.
            \item{}\label{besov_embedding} We have
        \[
            B^{s_{0}}_{p_0,q_0}(\Rl^d_\theta)\subseteq B^{s_1}_{p_1,q_1}(\Rl^d_\theta)
        \]
        with a corresponding norm inequality independent of $\theta$ whenever $p_0 \leq p_1$ and 
        \[
            s_0 \geq s_1+d\left(\frac{1}{p_0}-\frac{1}{p_1}\right),\quad q_0\leq q_1.
        \]
        \end{enumerate}
    \end{lemma}
    \begin{proof}
        From Theorem \ref{nc_mikhlin},
        \[
            \|\Delta_ju\|_{L_p(\Rl^d_\theta)} \lesssim 2^{-jr}\|u\|_{W^r_p(\Rl^d_\theta)}.
        \]
        It follows that
        \[
            2^{js}\|\Delta_j u\|_{L_p(\Rl^d_\theta)} \lesssim 2^{-j(r-s)}\|u\|_{W^{r}_p(\Rl^d_\theta)}.
        \]
        Therefore,
        \[
            \|u\|_{B^s_{p,q}(\Rl^d_\theta)} \lesssim \left(\sum_{j=0}^\infty 2^{-jq(r-s)}\right)^{\frac1q}\|u\|_{W^r_p(\Rl^d_\theta)}.
        \]
        This proves \eqref{sobolev_inclusion}.
    
        Next we prove \eqref{schwartz_dense}. Given $u \in B^s_{p,q}(\Rl^d_\theta),$ define
        \[
            u_N = \sum_{k=0}^N \Delta_ku.
        \]
        Since $q< \infty,$ we have
        \[
            \|u-u_N\|_{B^s_{p,q}} \lesssim \left(\sum_{k=N-1} 2^{ksq}\|\Delta_ku\|_{p}^q\right)^{\frac1q}
        \]
        which vanishes as $N\to \infty.$ Thus it suffices to show that there exists a sequence $\{\psi_j\}_{j=0}^\infty\subset\Sc(\Rl^d_\theta)$
        such that $\|u_N-\psi_j\|_{B^s_{p,q}(\Rl^d_\theta)}\to 0.$ Since the Fourier transform of $u_N$ is finitely supported, we 
        have 
        \[
            u_N \in W^{r}_p(\Rl^d_\theta)
        \]
        for all $r>0.$ Taking $r>s,$ there exists a sequence $\{\psi_j\}_{j=0}^\infty$ such that 
        \[
            \|\psi_j-u_N\|_{W^r_p(\Rl^d_\theta)} \to 0.
        \]
        (see \cite[Proposition 3.14]{MSX2}).
        Applying \eqref{sobolev_inclusion}, it follows that
        \[
            \|\psi_j-u_N\|_{B^s_{p,q}(\Rl^d_\theta)}\to 0.
        \]
        This proves \eqref{schwartz_dense}.
    
        The second and third inclusions are immediate from the Definition \ref{besov_spaces_definition}.
        The embedding and norm inequality in \eqref{besov_embedding} is an immediate consequence of the definition and \eqref{bernstein_for_littlewood_paley}.
    \end{proof}

    The behaviour of Besov spaces on $\Rl^d$ under interpolation is well-known \cite[Theorem 6.4.5]{Bergh-Lofstrom-1976}. These Besov spaces on $\Rl^d_\theta$ behave under real interpolation identically to the classical Besov spaces on $\Rl^d$. The relevant interpolation result is as follows.
    The result, and its proof, is basically identical to \cite[Proposition 5.1]{XXY2018}. 
    \begin{theorem}\label{besov_interpolation}
        The Besov spaces satisfy the following real interpolation relation
        \begin{equation*}
            (B^{\alpha_0}_{p,q_0}(\Rl^d_\theta),B^{\alpha_1}_{p,q_1}(\Rl^d_\theta))_{\eta,q} = B^{(1-\eta)\alpha_0+\eta\alpha_1}_{p,q}(\Rl^d_\theta)
        \end{equation*}
        Here, $1\leq q_0,q_1,q\leq\infty$, $1\leq p \leq\infty$, $\eta \in (0,1)$ and $\alpha_1\neq \alpha_2 \in \Rl$.
    \end{theorem}
    \begin{proof}
        For a Banach space $X$, we denote by $\ell_q^{s}(X)$ the 
        space of sequences $\{x_n\}_{n\geq 0}$ in $X$ such that
        \begin{equation*}
            \left(\sum_{n=0}^\infty 2^{ns q}\|x_n\|_X^q\right)^{1/q} < \infty
        \end{equation*}
        with suitable modifications for $q = \infty$. Under real interpolation, the spaces $\ell_q^s(X)$ behave as
        $$
            (\ell_{q_0}^{s_0}(X),\ell_{q_1}^{s_1}(X))_{\eta,q} = \ell^s_q(X)
        $$
        where $s = (1-\eta)s_0+\eta s_1$ and $0< q_1,q_2,q \leq \infty$. See \cite[Theorem 5.6.1]{Bergh-Lofstrom-1976}.    
        
        The mapping
        \begin{equation*}
            T(x) =  \{\Delta_j x\}_{j=0}^\infty
        \end{equation*}
        is an isometry from $B^s_{p,q}(\Rl^d_\theta)$ to $\ell_q^s(L_p(\Rl^d_\theta))$.
        In fact, $T$ is a retract of $\ell^s_q(L_p(\Rl^d_\theta))$. We can see 
        this by noting that on the image $T(B^s_{p,q}(\Rl^d_\theta))$ we can define a mapping
        \begin{equation*}
            P(\{x_n\}_{n\geq 0}) = S_0(x_0)+ \sum_{j=0}^\infty (\Delta_{j-1}+\Delta_j+\Delta_{j+1})(x_j)
        \end{equation*}
        which satisfies $PT(x) = x$ for all $x \in B^s_{p,q}(\Rl^d_\theta)$. 
        
        This allows us to deduce the interpolation theory from
        interpolation properties of $\ell_q^s(L_p)$, identically to \cite[Theorem 6.4.5]{Bergh-Lofstrom-1976}.
    \end{proof}
    
    
    Generalising Proposition \ref{approximation_facts}.\eqref{equivalence_of_B_22}, we have the following Littlewood-Paley characterisation
    of the $L_2$-Sobolev spaces.
    \begin{theorem}
        For $s \in \Rl$, we have $B^s_{2,2}(\Rl^d_\theta) = W^s_2(\Rl^d_\theta)$, with an equivalence of norms,
        \begin{equation*}
            \|u\|_{B^s_{2,2}(\Rl^d_\theta)} \lesssim_{s,d} \|u\|_{W^s_2(\Rl^d_\theta)} \lesssim_{s,d} \|u\|_{B^s_{2,2}(\Rl^d_\theta)}.
        \end{equation*}
    \end{theorem}
    A Littlewood-Paley characterization of $L_p$-Sobolev spaces for $1<p<\infty$ was obtained in \cite{MSX3}, but it is not relevant here.
    
    The mapping properties of the heat and Schr\"odinger semigroups on Sobolev and Besov spaces are deduced easily from Theorem \ref{heat_mapping}. Analogous
    results in the classical case are e.g. \cite[Theorem 5.29]{Sawano2018}
    \begin{proposition}\label{heat_mapping_besov}
        For $r,s \in \Rl$, $t > 0$ and $1\leq p \leq \infty$, we have
        \begin{equation*}
            \|e^{t\Delta}\|_{W^s_p(\Rl^d_\theta)\to W^r_p(\Rl^d_\theta)} \lesssim 1+t^{\frac{s-r}{2}}
        \end{equation*}
        with $\|e^{t\Delta}\|_{W^s_p\to W^s_p} \leq 1$.
        We also have 
        \begin{equation*}
            \|e^{\ri t\Delta}\|_{W^s_2(\Rl^d_\theta)\to W^s_2(\Rl^d_\theta)} \lesssim 1.
        \end{equation*}
        For $p,q \in [1,\infty]$, we have
        \begin{equation*}
            \|e^{t\Delta}\|_{B^s_{p,q}(\Rl^d_\theta)\to B^r_{p,q}(\Rl^d_\theta)} \lesssim 1+t^{\frac{s-r}{2}}.
        \end{equation*}
    \end{proposition}
    
    \begin{corollary}
        If $p \in [1,\infty]$ and $q \in [1,\infty)$, then $\Delta$ generates a contractive $C_0$-semigroup on $B^s_{p,q}(\Rl^d_\theta)$ for every $s \in \Rl$.
        The same is true with $q = \infty$ under the condition that $s > 0$.
    \end{corollary}
    \begin{proof}
        We only need to check that $\lim_{t\to 0} e^{t\Delta}u = u$ for all $u\in B^s_{p,q}(\Rl^d_\theta)$. This is equivalent to
        $$
            \lim_{t\to 0} \|e^{t\Delta}\Delta_ju-\Delta_ju\|_{L_p(\Rl^d_\theta)} = 0,\quad j\geq 0.
        $$
        which follows from Theorem \ref{heat_mapping}.
    \end{proof}
    
%

\section{Elementary pseudodifferential operators}\label{psdo_section}
    There is a well-developed theory of pseudodifferential operators for $\Rl^d_\theta$ \cite{GJP2017,GJM}. However, the existing calculus does not
    contain the operators we need for the study of nonlinear equations. For this reason, we introduce a new class of pseudodifferential operators for $\Rl^d_\theta$.
    Recall that a function $\sigma \in C^\infty(\Rl^d\times \Rl^d)$ is said to belong to the symbol class $S^0_{1,1}(\Rl^d\times \Rl^d)$ if for all $\alpha,\beta \in \Ntrl^d$ we have
    $$
        |\partial^{\alpha}_x\partial^{\beta}_\xi\sigma(x,\xi)| \lesssim_{\alpha,\beta} (1+|\xi|)^{|\alpha|-|\beta|}.
    $$    
    Borrowing the terminology of \cite[Chapter 13, Section 9]{Taylor-pde-3-2011}, a symbol $\sigma \in C^\infty(\Rl^d\times \Rl^d)$ is called elementary if there exists an expansion
    $$
        \sigma(x,\xi) = \sum_{j=0}^\infty a_j(x)\Psi_j(\xi)
    $$
    for some family of bounded functions $a_j$. The condition that $\sigma$ belongs to the symbol class $S^0_{1,1}(\Rl^d\times \Rl^d)$ is equivalent to
    $$
        \|\partial^{\alpha}a_j\|_{\infty} \lesssim_{\alpha} 2^{j|\alpha|}
    $$
    for all $\alpha\in \Ntrl^d$.
    We will not attempt to develop a complete theory of pseudodifferential operators on $\Rl^d_\theta$. Instead, we will only consider a noncommutative analogy of elementary operators.
    \begin{definition}\label{elementary_definition}
        Let $a = \{a_j\}_{j=0}^\infty$ and $b = \{b_j\}_{j=0}^\infty$ be sequences of elements of $L_\infty(\Rl^d_\theta)$ such that for all $\alpha \in \Ntrl^d$
        we have
        \begin{equation*}
            \|\partial^{\alpha}a_j\|_\infty \lesssim_{\alpha} 2^{j|\alpha|},\quad \|\partial^{\alpha}b_j\|_{\infty} \lesssim_{\alpha} 2^{j|\alpha|}.
        \end{equation*}
        An \emph{elementary pseudodifferential operator} is a linear operator $T_{a,b}:\lt(C^\infty_c(\Rl^d)) \to \Sc'(\Rl^d_\theta)$ given by the formula
        \begin{equation*}
            T_{a,b}(u) = \sum_{j=0}^\infty a_j\Delta_j(u)b_j.
        \end{equation*}
    \end{definition}
    The difference between these operators $T_{a,b}$ and the operators in \cite{GJP2017,GJM} is that we consider symbols that act on both the left and on the right.    
    If we have $b_j = 1$ for all $j\geq 0$, then the operator $T_{a,b}$ is a pseudodifferential operator with symbol in the class $S^{0}_{1,1}$ in the setting of \cite{GJP2017}. 
    Note that we have only defined $T_{a,b}(u)$ for $u = \lt(f)$ for some compactly supported smooth $f$. Once the continuity of $T_{a,b}$ on Besov spaces
    is established, the operator may be extended by continuity.
    
%
%
%
    To estimate the norms of the operators $T_{a,b}$ on Besov spaces, it is convenient to have the following notation:
    \begin{definition}
        Let $a = \{a_j\}_{j=0}^\infty$ and $b = \{b_j\}_{j=0}^\infty$ be sequences of elements of $L_\infty(\Rl^d_\theta)$ as in Definition \ref{elementary_definition}
        For $k\geq 0$, denote by $M_k(a)$ and $M_k(b)$ the following quantities:
        \begin{equation*}
            M_{k}(a) := \sup_{j\geq 0, |\alpha|\leq k} 2^{-|\alpha|j}\|D^{\alpha}a_j\|_{\infty} < \infty,\quad M_{k}(b) := \sup_{j\geq 0, |\alpha|\leq k} 2^{-|\alpha|j}\|D^{\alpha}a_j\|_{\infty} < \infty
        \end{equation*}
    \end{definition}
    
    The main result of this section is that elementary pseudodifferential operators are bounded on Besov spaces of positive regularity.
    \begin{theorem}\label{besov_mapping_theorem}
        For all $p,q\in [1,\infty]$ and $s> 0$, we have that $T_{a,b}:\lambda_\theta(C^\infty_c(\Rl^d_\theta))\to \Sc'(\Rl^d_\theta)$ extends to a continuous linear map
        $$
            T_{a,b}:B^s_{p,q}(\Rl^d_\theta)\to B^s_{p,q}(\Rl^d_\theta)
        $$
        with norm
        $$
            \|T_{a,b}\|_{B^{s}_{p,q}\to B^s_{p,q}} \lesssim_{s,p,q,d} M_{s+2}(a)M_{s+2}(b).
        $$
    \end{theorem}
    The proof of Theorem \ref{besov_mapping_theorem} follows the same lines as similar results in other settings. In particular, the proof of \cite[Lemma 5.1]{Xia-Xiong-mapping-properties}
    is quite similar.
    
    \begin{lemma}\label{dyadic_block_lemma}
        Let $T_{a,b}$ be an elementary pseudodifferential operator. For all multi-indices $\alpha\in \Ntrl^d$ and $p\in [1,\infty]$, we have
        $$
            \|D^{\alpha}T_{a,b}\Delta_ju\|_{p} \lesssim 2^{j|\alpha|}M_{|\alpha|}(a)M_{|\alpha|}(b), \quad j\geq 0.
        $$
    \end{lemma}
    \begin{proof}
        We have that $\Delta_j\Delta_k = 0$ if $|j-k|\geq 2$. Therefore,
        \begin{equation*}
            D^{\alpha}T_{a,b}\Delta_ju = \sum_{k=j-1}^{j+1} D^{\alpha}(a_j(\Delta_k\Delta_j u)b_j).
        \end{equation*}
        Using the Leibniz rule and Proposition \ref{homogeneous_bound}, the result follows.        
    \end{proof}
    
    \begin{lemma}\label{key_lemma_for_mapping}
        Let $s \geq 0$. Then for all $j,k\geq 0$ and $p \in [1,\infty]$, we have
        $$
            2^{sj}\|\Delta_jT_{a,b}\Delta_ku\|_p\lesssim 2^{sk}\|\Delta_k u\|_pM_{s+1}(a)M_{s+1}(b),\quad j\geq 0.
        $$
    \end{lemma}
    \begin{proof}
        
        Let $\psi$ be a smooth function supported in a ball of radius $1/2$ centred at zero, and equal to $1$ near zero. Define
        \begin{equation*}
            a_{0,0}(\xi) = \psi(\xi), a_{0,j}(\xi) = (1-\psi(\xi))\frac{\xi_j}{|\xi|^2},\quad j=1,\ldots,d,\,\xi \in \Rl^d.
        \end{equation*}
        Then we have
        \begin{equation*}
            1 = \sum_{j = 0}^d a_{0,j}(\xi)\xi_j.
        \end{equation*}
        If we raise this identity to the power $l\geq 1$, we construct smooth bounded functions $a_{\gamma,l}$ for $\gamma \in \Ntrl^d$ such that
        \begin{equation*}
            1 = \sum_{|\gamma|\leq l} a_{\gamma,l}(\xi)\xi^{\gamma},\quad \xi \in \Rl^d.
        \end{equation*}
        By design, the functions $a_{\gamma,l}$ are homogeneous of order $-|\gamma|$ outside the ball of radius $1/2$.
        Specifically, we have
        \begin{equation*}
            a_{\gamma,l}(\xi) = \psi(\xi)^{l-|\gamma|}(1-\psi(\xi))^{|\gamma|}\frac{\xi^\gamma}{|\xi|^{2|\gamma|}}.
        \end{equation*}
        Let $P_{\gamma,l}$ be the operator
        \begin{equation*}
            P_{\gamma,l} = a_{\gamma,l}(D).
        \end{equation*}
        Therefore,
        \begin{equation}\label{partition_of_one}
            1 = \sum_{|\gamma|\leq l} P_{\gamma,l}D^\gamma.
        \end{equation}
        Now we write
        \begin{align*}
            2^{sj}\|\Delta_j T_{a,b}\Delta_ku\|_p &\leq 2^{sj}\sum_{|\gamma|\leq l} \|\Delta_j P_{\gamma,l}D^{\gamma}T_{a,b}\Delta_k u\|_p\\
                                                  &\leq 2^{sj} \sum_{|\gamma|\leq l}\|\Delta_jP_{\gamma,l}\|_{L_p(\Rl^d_\theta)\to L_p(\Rl^d_\theta)}\|D^{\gamma}T_{a,b}\Delta_ku\|_p.
        \end{align*}
        Using Lemma \ref{dyadic_block_lemma} and Proposition \ref{homogeneous_bound}, we have
        \begin{align*}
            2^{sj}\|\Delta_j T_{a,b}\Delta_ku\|_p &\lesssim 2^{sj} \cdot 2^{-jl}\cdot 2^{kl}\sup_{|\gamma|,|\delta|\leq l}M_{\gamma}(a)M_{\delta}(b)\|\Delta_k u\|_p\\
                                                  &= 2^{ks} 2^{(s-l)(j-k)} \sup_{|\gamma|,|\delta|\leq l} M_{\gamma}(a)M_{\delta}(b)\|\Delta_k u\|_p.
        \end{align*}
        If $j \leq k$, then we can take $l=0$ and we are done. Otherwise, choose $l > s$. Then we have $(s-l)(j-k) < 0$, and
        \begin{equation*}
            2^{sj}\|\Delta_j T_{a,b}\Delta_ku\|_p \lesssim 2^{sk}\sup_{|\gamma|,|\delta|\leq s+1} M_{\gamma}(a)M_\delta(b)\|\Delta_k u\|_p.
        \end{equation*}
    \end{proof}
    
    \begin{proposition}\label{exotic_boundedness}
        Let $T_{a,b}$ be as in the statement of Theorem \ref{besov_mapping_theorem}, and let
        $1\leq p \leq \infty$. The operator $T_{a,b}$ has the following mapping properties:
        \begin{enumerate}[{\rm (i)}]
            \item{}\label{zero_smoothness} $T:B^0_{p,1}(\Rl^d_\theta)\to L_p(\Rl^d_\theta)$ continuously, with norm at most a constant multiple of $M_{0}(a)M_0(b)$.
            \item{}\label{positive_smoothness} If $s > 0$ then $T:B^s_{p,1}(\Rl^d_\theta) \to B^s_{p,\infty}(\Rl^d_\theta)$ continuously, with norm at most a constant multiple of $M_{s+1}(a)M_{s+1}(b)$.
        \end{enumerate}
    \end{proposition}
    \begin{proof}
        To prove \eqref{zero_smoothness}, let $u \in B^0_{p,1}(\Rl^d_\theta)$. Using the $L_p$-triangle inequality and the Littlewood-Paley decomposition, we have
        $$
            \|T_{a,b}u\|_{p} \leq \sum_{j=0}^\infty \|T_{a,b}(\Delta_{j} u)\|_p
        $$
        Since $\Delta_j\Delta_k = 0$ unless $|j-k|\leq2$, this sum further decomposes as
        \begin{align*}
            \|T_{a,b}u\|_p &\leq\sum_{j=0}^\infty \sum_{|j-k|\leq 2} \|a_{k}(\Delta_j(\Delta_k u))b_k\|_p\\
                           &\leq \sum_{j=0}^\infty \sum_{|j-k|\leq 2}\|a_k\|_\infty\|b_k\|_\infty \|\Delta_j(u)u\|_p\\
                           &\lesssim M_0(a)M_0(b)\|u\|_{B^0_{p,1}}.
        \end{align*} 

        Now we prove \eqref{positive_smoothness}. Initially suppose that $u \in \Sc(\Rl^d_\theta)$. Let $j\geq 0$. Using Lemma \ref{key_lemma_for_mapping}, we have
        \begin{align*}
            2^{sj}\|\Delta_j T_{a,b}u\|_p &\leq 2^{sj} \sum_{k=0}^\infty 2^{sj}\|\Delta_j T_{a,b}\Delta_k u\|_p\\
                                          &\lesssim \sum_{k=0}^\infty 2^{sk} \|\Delta_k u\|_p \sup_{|\alpha|,|\beta|\leq s+1} M_{\alpha}(b)M_{|\beta|}(b)\\
                                          &\lesssim \|u\|_{B^{s}_{p,1}}\sup_{|\alpha|,|\beta|\leq s+1}M_{|\alpha|}(a)M_{|\beta|}(b).
        \end{align*}
        Taking the supremum over $j\geq 0$ yields
        $$
            \|T_{a,b}(u)\|_{B^s_{p,\infty}} \lesssim M_{s+1}(a)M_{s+1}(b)\|u\|_{B^s_{p,1}}.
        $$
        Since $\Sc(\Rl^d_\theta)$ is dense in $B^s_{p,1}(\Rl^d_\theta)$, the result follows by continuity.
    \end{proof}
    
    Using interpolation, we can complete the proof of the Besov mapping properties of elementary pseudodifferential operators.
    \begin{proof}[Proof of Theorem \ref{besov_mapping_theorem}]
        Using Proposition \ref{exotic_boundedness}, we have
        \begin{align*}
            T_{a,b}&:B^0_{p,1}(\Rl^d_\theta)\to L_p(\Rl^d_\theta) \subset B^0_{p,\infty}(\Rl^d_\theta),\\
            T_{a,b}&:B^{s+1}_{p,1}(\Rl^d_\theta)\to B^{s+1}_{p,\infty}(\Rl^d_\theta)
        \end{align*}
        with norms at most a constant multiple of $M_{s+2}(a)M_{s+2}(b)$. Applying Theorem \ref{besov_interpolation}, it follows that for every $1\leq q\leq \infty$ we have
        $$
            T_{a,b}:B^{s}_{p,q}(\Rl^d_\theta)\to B^{s}_{p,q}(\Rl^d_\theta)
        $$
        with norm at most a constant multiple of $M_{s+2}(a)M_{s+2}(b)$. 
    \end{proof}
    
    \begin{remark}
        The bound $M_{s+2}(a)M_{s+2}(b)$ in the proof of Theorem \ref{besov_mapping_theorem} above is not optimal. The same proof yields an upper bound of $M_{s+1+\varepsilon}M_{s+1+\varepsilon}(b)$ for every $\varepsilon> 0$.
    \end{remark}

\section{Multiplication on Besov spaces}\label{multiplication_section}
    We now study the product $(u,v)\mapsto uv$ on Besov spaces $B^s_{p,q}(\Rl^d_\theta)$. In the classical theory, 
    one approach to this problem is the so-called Bony decomposition.

    Given $f,g \in  \Sc(\Rl^d)$, there exists $f\at g \in \Sc(\Rl^d)$ such that
    \begin{equation*}
        \lt(f)\lt(g) = \lt(f\at g).
    \end{equation*}
    An important feature of the $\theta$-convolution $\at$ is that we have
    $$
        \mathrm{supp}(f\at g)\subseteq \mathrm{supp}(f)+\mathrm{\supp(g)}.
    $$
    It follows that if $j,k\in \Itgr$, then $(\Delta_jx)(\Delta_k y)$ has is the image under $\lt$ of a function supported 
    in the ball of radius $2^{j+1}+2^{k+1} \leq 2^{\max\{j,k\}+2}.$ If 
    $|j-k|> 2$, then $(\Delta_jx)(\Delta_ky)$ has is the image under $\lt$ of a function with support contained in the annulus
    \begin{equation*}
        \{t\in \Rl^d\;:\;2^{|j-k|-2} \leq |t| \leq 2^{\max\{j,k\}+2}\}.
    \end{equation*}
    We can summarise these observations in the following lemma.
    \begin{lemma}
        Let $x,y \in \Sc(\Rl^d_\theta)$ and let $j,k\in \geq 0$. If $l \geq \max\{j,k\}+3$, then
        \begin{equation*}
            \Delta_{l}(\Delta_j(x)\Delta_k(y)) = 0.
        \end{equation*}
        Moreover, if $|j-k|>2$, then the same holds for all $l \leq |j-k|-3$.
    \end{lemma}
    
    Let $u,v \in \Sc(\Rl^d_\theta)$. At least formally, we have the Bony decomposition \cite[Section 2.8.1]{Bahouri-Chemin-Danchin-2011}.
    $$
        uv = \sum_{j,k\geq 0} \Delta_j(u)\Delta_k(v) = \sum_{j\leq k-3} \Delta_j(u)\Delta_k(v)+ \sum_{|j-k|\leq 2} \Delta_j(u)\Delta_k(v) + \sum_{j\geq k+3} \Delta_{j}(u)\Delta_k(v).
    $$
    We denote this as
    \begin{equation*}
        uv = \pilo(u,v) + \res(u,v) + \pihi(u,v)
    \end{equation*} 
    where the low and high frequency ``paraproducts" $\pilo(u,v), \pihi(u,v)$ and the ``resonating term" $\res(u,v)$ are defined as
    \begin{align*}
           \pilo(u,v) &:= \sum_{j\geq 3} S_{j-3}(u)\Delta_j(v),\\
           \pihi(u,v) &:= \sum_{j\geq 3} \Delta_j(u)S_{j-3}(v),\\
            \res(u,v) &:= \sum_{j,k\geq 0, |j-k|\leq 2} \Delta_j(u)\Delta_k(v).
    \end{align*}    
    By design, the paraproducts $\pilo(u,\cdot)$ and $\pihi(\cdot,v)$ are elementary pseudodifferential operators. Using Theorem \ref{besov_mapping_theorem}, we deduce the following.
    \begin{theorem}
        Let $u,v \in \Sc(\Rl^d_\theta)$. If $s > 0$ and $p,q \in [1,\infty]$, we have the following:
        \begin{enumerate}[{\rm (i)}]
            \item{} $\|\pilo(u,v)\|_{B^s_{p,q}} \lesssim \|u\|_\infty\|v\|_{B^s_{p,q}},$
            \item{} $\|\pihi(u,v)\|_{B^s_{p,q}} \lesssim \|u\|_{B^s_{p,q}}\|v\|_\infty.$
        \end{enumerate}
    \end{theorem}
    \begin{proof}
        By Proposition \ref{approximation_facts}, we have
        $$
            \sup_{j\geq 0,\,\alpha\in \Ntrl^d} 2^{-j|\alpha|}\|D^{\alpha}S_{j-3}(u)\|_\infty \lesssim \|u\|_\infty
        $$
        and
        $$
            \sup_{k\geq 0,\,\alpha\in \Ntrl^d} 2^{-k|\alpha|}\|D^{\alpha}S_{k-3}(v)\|_\infty \lesssim \|v\|_\infty.
        $$
        Theorem \ref{besov_mapping_theorem} immediately yields the boundedness of $v\mapsto \pilo(u,v)$ and $u\mapsto \pihi(u,v)$ on every Besov
        class $B^s_{p,q}(\Rl^d_\theta)$ with $s > 0$, with norm bounds $\|u\|_\infty$ and $\|v\|_{\infty}$ respectively.
    \end{proof}
    
    The resonating term $\res(u,v)$ can also be considered as an elementary pseudodifferential operator if we fix one argument.
    That is, the mapping $v\mapsto \res(u,v)$ is elementary pseudodifferential.
    With some more effort it is possible to show that the resonating term $\res(u,v)$ has better regularity than $u$ and $v$ individually
    \begin{theorem}
        If $s_0,s_1 \in \Rl$ satisfy $s_0+s_1> 0$, then
        $$\|\res(u,v)\|_{B^{s_0+s_1}_{p,q}} \leq \|u\|_{B^{s_0}_{p_0,q_0}}\|v\|_{B^{s_0}_{p_0,q_0}}$$
        where $p^{-1} = p_0^{-1}+p_1^{-1}$ and $q^{-1} = q_0^{-1}+q_1^{-1}$.
    \end{theorem}
    \begin{proof}
        If $|n-m|\leq 2$, then we have $\max\{n,m\}\leq n+2$ and thus
        \begin{equation*}
            \Delta_{k}(\Delta_n(u)\Delta_m(v)) = 0\text{ if }k \geq n+3.
        \end{equation*}
        Hence,
        \begin{equation*}
            \Delta_n(\res(u,v)) = \sum_{|j-k|\leq 2, j,k\geq n-3} \Delta_n(\Delta_{j}(u)\Delta_k(v)).
        \end{equation*} 
        It follows that if $s:=s_0+s_1 > 0$ then
        \begin{align*}
            \|\res(u,v)\|_{B^s_{p,q}} &\lesssim \max_{|\nu|\leq 2}\left(\sum_{n=0}^\infty 2^{nsq}\|\Delta_n(\res(u,v))\|_p^q\right)^{1/q}\\
                                      &\leq \max_{|\nu|\leq 2}\left(\sum_{n=0}^\infty 2^{nsq}\left\|\sum_{k\geq n-3} \Delta_n(\Delta_k u \cdot \Delta_{k+\nu}(v))\right\|_p^q\right)^{1/q}\\
                                      &\lesssim \max_{|\nu|\leq 2}\left(\sum_{n=0}^\infty \left(\sum_{k\geq n-3} 2^{(n-k)s}2^{ks}\|\Delta_{k}(u)\Delta_{k+\nu}(v)\|_p\right)^{q}\right)^{1/q}\\
                                      &\lesssim \max_{|\nu|\leq 2}\left(\sum_{k\geq \max\{\nu,0\}} 2^{ksq}\|\Delta_k(u)\Delta_{k-\nu}(v)\|_{p}^q\right)^{1/q}
        \end{align*}
        where the second-to-final line is an application of Young's convolution inequality, using the assumption that $s>0.$ Hence,
        \[
            \|\res(u,v)\|_{B^s_{p,q}} \lesssim \max_{|\nu|\leq 2}\left(\sum_{k\geq \max\{\nu,0\}} 2^{ksq}\|\Delta_k(u)\Delta_{k-\nu}(v)\|_{p}^q\right)^{1/q}.
        \]
        Applying the H\"older inequality immediately yields the result.
    \end{proof}
    
    Writing $uv = \pilo(u,v)+\res(u,v)+\pihi(u,v)$ yields the following product estimate:
    \begin{corollary}\label{product_estimate}
        If $s > 0$ and $p,q\in [1,\infty]$, then for all $u,v \in B^s_{p,q}(\Rl^d_\theta)\cap L_\infty(\Rl^d_\theta)$, we have $uv \in B^s_{p,q}(\Rl^d_\theta)\cap L_\infty(\Rl^d_\theta)$
        with the norm bound
        \begin{equation*}
            \|uv\|_{B^{s}_{p,q}} \lesssim_{s,p,q} \|u\|_{B^{s}_{p,q}}\|v\|_\infty+\|v\|_\infty\|y\|_{B^s_{p,q}}.
        \end{equation*}
    \end{corollary}
    Classical analogies of this result are well-known. See \cite[Chapter 4]{Runst-Sickel-1996}, \cite[Theorem 4.36]{Sawano2018}, \cite[Corollary 2.54]{Bahouri-Chemin-Danchin-2011}.

    While irrelevant for the applications in this paper, for the sake of completeness we also include results for the product $uv$
    when one of $u$ or $v$ has negative regularity. 
    
    The following is an immediate consequence of the definition.
    \begin{lemma}
        Let $p,q \in [1,\infty]$. If $s < 0$, then
        \begin{equation*}
            \|S_jx\|_p \lesssim 2^{-s j}\|x\|_{B^s_{p,q}},\quad j\geq 0.
        \end{equation*}
    \end{lemma}

    \begin{corollary}
        Let $s_0 < 0 < s_1 \in \Rl$ be such that $s_0+s_1 > 0$, and suppose that $p_0,q_0,p_1,q_1 \in [0,\infty]$ are chosen such that
        $$
            \frac{1}{p} = \frac{1}{p_0}+\frac{1}{p_1},\quad \frac{1}{q} = \frac{1}{q_0}+\frac{1}{q_1}.
        $$
        Then
        \begin{enumerate}[{\rm (i)}]
            \item{} $\|\pilo(u,v)\|_{B^{s_1}_{p,q}} \lesssim \|u\|_{B^{s_0}_{p_0,q_0}}\|v\|_{B^{s_1}_{p_1,q_1}},$ and
            \item{} $\|\pihi(u,v)\|_{B^{s_0}_{p,q}} \lesssim \|u\|_{B^{s_0}_{p_0,q_0}}\|v\|_{B^{s_1}_{p_1,q_1}}.$
        \end{enumerate}
    \end{corollary}
    It follows that if one of $s_0,s_1$ is negative and $s_0+s_1 > 0$, then
    $$
        \|uv\|_{B^{\min\{s_0,s_1\}}_{p,q}} \lesssim \|u\|_{B^{s_0}_{p_0,q_0}}\|v\|_{B^{s_1}_{p_1,q_1}}
    $$
    for all $u,v \in \Sc(\Rl^d_\theta)$.

\section{Nemytskij operators on $\Rl^d_\theta$}\label{nemytskij_section}
    Recall that in the classical setting, if $F$ is a smooth function on $\Rl$ and $u$
    belongs to some function space on $\Rl^d$, the nonlinear operation
    \begin{equation*}
        u\mapsto F(u)
    \end{equation*}
    is sometimes called a Nemytskij operator \cite{Runst-Sickel-1996}. We now proceed
    to study Nemytskij operators on $\Rl^d_\theta$ in terms of the strategy outlined in the introduction, combining the Meyer decomposition \eqref{meyer_decomposition_intro}
    and the L\"owner decomposition \eqref{lowner_formula_intro}.

\subsection{Birman-Solomyak classes}
    We say that a function $\phi$ of two variables belongs to the Birman-Solomyak class if $\phi$ has a suitable factorisation as an integral of products
    of functions of one variable.
    \begin{definition}
        Let $\phi:\Rl^2\to \Cplx$ be a Borel function on $\Rl^2$. Say that $\phi$ belongs to the Birman-Solomyak class $\BS$ if there exists a measure space $(\Omega,\mu)$
        with finite total variation
        and measurable functions $\alpha,\beta:\Rl\times \Omega\to \Cplx$ such that
        \begin{equation}\label{BS_decomp_def}
            \phi(t,s) = \int_{\Omega} \alpha(t,\omega)\beta(s,\omega)\,d\mu(\omega),\quad t,s \in \Rl
        \end{equation}
        and such that
        \begin{equation*}
            \int_{\Omega} \sup_{t \in \Rl}|\alpha(t,\omega)|\sup_{s\in \Rl}|\beta(s,\omega)|\,d|\mu(\omega)| < \infty.
        \end{equation*}
        Say that $\phi$ belongs to $\BS^\infty$ if the functions $\alpha$ and $\beta$ can be chosen to be smooth in the sense that all of the derivatives
        \begin{equation*}
            \frac{\partial^k}{\partial t^k}\alpha(t,\omega),\quad \frac{\partial^k}{\partial s^k}\beta(s,\omega)
        \end{equation*}
        exist, and for every $k,l\geq 0$ we have
        \begin{equation*}
            \int_{\Omega} \sup_{t\in \Rl}\left|\frac{\partial^k}{\partial t^k}\alpha(t,\omega)\right|\sup_{s \in \Rl} \left|\frac{\partial^l}{\partial s^l}\beta(s,\omega)\right|\,d|\mu|(s) < \infty.
        \end{equation*}
    \end{definition}
    
    The reason to consider this class is the following result:
    \begin{theorem}\cite{Peller-1985}
        Let $(\Mv,\tau)$ be a von Neumann algebra, and $X,Y\in \Mv.$ Assume that the difference $X-Y$ extends to an element of $L_p(\Mv,\tau)$ where $1\leq p \leq \infty$.
        Let $F$ be a Lipschitz function on $\Rl$. If the divided difference function
        \begin{equation*}
            F^{[1]}(t,s) := \frac{F(t)-F(s)}{t-s},\quad t\neq s\in \Rl.
        \end{equation*}
        belongs to $\BS$, then we have an estimate
        \begin{equation*}
            \|F(X)-F(Y)\|_{L_p(\Mv,\tau)} \leq C_F\|X-Y\|_{L_p(\Mv,\tau)}.
        \end{equation*}
        Here, the constant $C_F$ depends on $F$ but not on $X$ and $Y$.
                
    \end{theorem}
    This theorem has a partial converse: if $F$ obeys a Lipschitz estimate $\|f(A)-f(B)\|_{\infty} \lesssim \|A-B\|_{\infty}$ for all pairs
    $A$ and $B$ of operators whose spectral measures are absolutely continuous with respect to some reference Borel measure $\nu$ on $\Rl,$ then outside some set of $\nu$ measure zero, the divided difference $F^{[1]}$ admits a representation of the form \eqref{BS_decomp_def}. This converse result is due to Peller \cite{Peller-1985}. 
    
    The preceding theorem is based on the following computation, which will also be useful here:
    \begin{lemma}\label{lowner_formula}
        Let $\Mv$ be a von Neumann algebra represented on a separable Hilbert space. Let $X,Y \in \Mv$ be self-adjoint, and let $F:\Rl\to \Rl$ be a function such that the divided difference $F^{[1]}$ belongs to $\BS$, with the decomposition
        \begin{equation*}
            \frac{F(t)-F(s)}{t-s} = \int_{\Omega} \alpha(t,\omega)\beta(s,\omega)\,d\mu(\omega),\quad t\neq s\in \Rl.
        \end{equation*}
        Then
        \begin{equation*}
            F(X)-F(Y) = \int_{\Omega} \alpha(X,\omega)(X-Y)\beta(Y,\omega)\,d\mu(\omega).
        \end{equation*}
        where the integral converges in the weak$^*$-sense. For all $1\leq p \leq \infty$, we have the estimate
        \begin{equation*}
            \|F(X)-F(Y)\|_{L_p(\Mv,\tau)} \leq \|F^{[1]}\|_{\BS}\|F(X)-F(Y)\|_{L_p(\Mv,\tau)}.
        \end{equation*}
    \end{lemma} 
    \begin{proof}[(Sketch of proof)]
        A full proof of this identity requires some care due to certain technicalities regarding the convergence of the integral. The required technical details may be found in \cite[Section 4]{DDSZ-2020-II}.
        In place of a complete argument we provide a sketch proof which illustrates the basic idea.
        Let $X$ and $Y$ have spectral decompositions
        \begin{equation*}
            X = \int_{\Rl} t\,dE_X(t),\quad Y = \int_{\Rl} s\,dE_Y(s).
        \end{equation*}
        where $E_X$ and $E_Y$ are the spectral measures for $X$ and $Y$ respectively.
        Then at least formally we have
        \begin{align*}
            F(X)-F(Y) &= \int_{\Rl} F(t)\,dE_X(t)-\int_{\Rl} F(s)\,dE_Y(s)\\
                      &= \int_{\Rl}\int_{\Rl} F(t)-F(s) dE_X(t)dE_Y(s)\\
                      &= \int_{\Rl}\int_{\Rl} F^{[1]}(t,s)(t-s)dE_X(t)dE_Y(s)\\
                      &= \int_{\Rl}\int_{\Rl}\int_{\Omega} \alpha(t,\omega)\beta(s,\omega)\,d\mu(\omega) (t-s)dE_X(t)dE_Y(s)\\
                      &= \int_{\Rl}\int_{\Rl} \int_{\Omega} \alpha(t,\omega)\beta(s,\omega)\,d\mu(\omega)dE_X(t)(X-Y)dE_Y(s)\\
                      &= \int_{\Omega} \int_{\Rl} \alpha(t,\omega)dE_X(t)(X-Y)\int_{\Rl} \beta(s,\omega)\,dE_X(s)\,d\mu(\omega)\\
                      &= \int_{\Omega} \alpha(X,\omega)(X-Y)\beta(Y,\omega)\,d\mu(\omega).
        \end{align*}
        The norm bound on $F(X)-F(Y)$ can be seen heuristically from the triangle inequality applied to the above integral representation, although a careful argument based on interpolation
        may be found in \cite[Section 4]{DDSZ-2020-II}.
    \end{proof}
    
    There is no known analytic condition on a Lipschitz function $F$ which is both necessary and sufficient for $F^{[1]}\in \BS$. However, a result due to Peller states that
    $F \in \dot{B}^1_{\infty,1}(\Rl)$ (the homogeneous Besov space on $\Rl$) is sufficient \cite{Peller-1985}. 
    
    The following sufficient condition is far from being necessary, however since we restrict attention to smooth functions it will suffice.
    \begin{proposition}\label{Cc_is_BS}
        Let $F$ be a smooth compactly supported function on $\Rl$. Then $F^{[1]}\in \BS^\infty$.
    \end{proposition}
    \begin{proof}
        For $t\neq s\in \Rl$, we have the formula
        \begin{equation*}
            \frac{F(t)-F(s)}{t-s} = \int_0^1 F'((1-\eta)t+\eta s)\,d\eta.
        \end{equation*}        
        Let $g\in \Sc(\Rl)$ be the Fourier transform of $F'$. Then by the Fourier inversion formula we have
        \begin{equation*}
            F^{[1]}(t,s) = (2\pi)^{-\frac{d}{2}}\int_{0}^1\int_{-\infty}^\infty e^{\ri\xi(1-\eta)t}e^{\ri\xi\eta s}g(\xi)\,d\xi d\eta,\quad t\neq s \in \Rl.
        \end{equation*}
        This is a Birman-Solomyak decomposition, with $\Omega = [0,1]\times \Rl$ and
        $$
            d\mu(\eta,\xi) = g(\xi)d\eta\,d\xi,\quad \alpha(t,(\eta,\xi)) = e^{\ri\xi(1-\eta)t}\text{ and }\beta(s,(\eta,\xi)) = e^{\ri \xi\eta s}.
        $$
        Since $g$ is Schwartz class, it is easily verified that $F^{[1]} \in \BS^\infty$.
    \end{proof}


%
%

    \begin{corollary}\label{nonlinear_continuity}
        Let $F \in C^\infty(\Rl)$, and let $u = u^* \in L_p(\Rl^d_\theta)\cap L_\infty(\Rl^d_\theta)$ for some $p < \infty$. Then
        \begin{equation*}
            \lim_{j\to\infty} \|F(S_ju)-F(u)\|_p = 0.
        \end{equation*}
        The same holds for $p = \infty$ provided that $u \in C_0(\Rl^d_\theta)$ or $B^{s}_{\infty,\infty}(\Rl^d_\theta)$ for some $s > 0$.
    \end{corollary}
    \begin{proof}
        Using Propositions \ref{homogeneous_bound} and  \ref{approximation_facts}, we have
        $$
            \sup_{j\geq 0} \|S_ju\|_\infty \leq \|u\|_\infty.
        $$
        Thus by modifying $F$ outside the interval $[-\|u\|_\infty,\|u\|_\infty]$ if necessary, 
        we may assume without loss of generality that $F$ is compactly supported. Lemma \ref{Cc_is_BS} implies that
        $$
            \|F(S_ju)-F(u)\|_p \leq C_F\|S_ju-u\|_p,\quad j\geq 0.
        $$
        This vanishes as $j\to\infty$, due to Proposition \ref{approximation_facts}.
    \end{proof}
    
\subsection{Noncommutative Meyer decomposition}
    Let $F \in C^\infty(\Rl)$, and let $u \in \Sc(\Rl^d_\theta)$ be self-adjoint.
    
    We want to study the Nemytskij operator $u\mapsto F(u)$. Initially assume that $u = S_Nu$ for some $N\geq 0$, so that there can be no doubt that we have
    $$
        F(u) = \lim_{j\to \infty} F(S_ju).
    $$
    
    It follows that we have the series representation
    \begin{equation*}
        F(u) = F(S_0u)+\sum_{j=1}^\infty F(S_ju)-F(S_{j-1}u).
    \end{equation*}
    Indeed, since we assume that $u$ is the image under $\lt$ of a compactly supported function, this series actually terminates.
    Assume that $F^{[1]} \in \BS^\infty$ has the decomposition
    \begin{equation*}
        F^{[1]}(t,s) = \int_{\Omega} \alpha(t,\omega)\beta(s,\omega)\,d\mu(\omega),\quad t,s\in \Rl.
    \end{equation*} 
    Then using Lemma \ref{lowner_formula} to represent each $F(S_ju)-F(S_{j-1}u)$, we have
    \begin{align*}
        F(u) &= F(S_0u) + \sum_{j=1}^\infty \int_{\Omega} \alpha(S_ju,\omega)(S_ju-S_{j-1}u)\beta(S_{j-1}u,\omega)\,d\mu(\omega)\\
             &= F(S_0u) + \sum_{j=1}^\infty \int_{\Omega} \alpha(S_{j}u,\omega)(\Delta_ju)\beta(S_{j-1}u,\omega)\,d\mu(\omega)\\
             &= F(S_0u) + \int_{\Omega} \sum_{j=1}^\infty \alpha(S_ju,\omega)(\Delta_j u)\beta(S_{j-1}u,\omega)\,d\mu(\omega).
    \end{align*}
    Since $F$ is smooth, there exists a smooth function $G$ such that
    $$
        F(t) = F(0)+tG(t),\quad t \in \Rl.
    $$
    Replacing $F(S_0u)$ with $F(0)+G(S_0u)S_0u$, it follows that
    \begin{align}   
        F(u) &= F(0)+G(S_0u)S_0u+\int_{\Omega} \sum_{j=1}^\infty \alpha(S_ju,\omega)(\Delta_j u)\beta(S_{j-1}u,\omega)\,d\mu(\omega)\nonumber\\
             &= F(0)+G(S_0u)\Delta_0u+ \int_{\Omega} T_{\omega}(u)\,d\mu(\omega).\label{meyer_decomposition}
    \end{align}
    Here, $T_{\omega}$ is the linear operator on $\lt(C^\infty_c(\Rl^d))$ given by
    \begin{equation*}
        T_\omega x = \sum_{j=1}^\infty \alpha(S_ju,\omega)(\Delta_j x)\beta(S_{j-1}u,\omega),\quad \omega \in \Omega,\, x \in \lt(C^\infty_c(\Rl^d)).
    \end{equation*}
    Therefore the mapping properties of $u\mapsto F(u)$ are essentially reduced to studying the (linear) operator $T_\omega$
    and the remainder term $x\mapsto G(S_0u)\Delta_0(x)$. In order to apply Theorem \ref{besov_mapping_theorem}, we will prove that $T_\omega$ is an elementary pseudodifferential operator.

    In order to prove that the operator $T_\omega$ is elementary pseudodifferential, we will need the following general assertion:
    \begin{theorem}\label{smooth_coefficient}
        Let $G:\Rl\to \Cplx$ be a smooth function, and let $u^* = u \in L_\infty(\Rl^d_\theta)$. Then for all $\alpha\in \Ntrl^d$,
        \begin{equation*}
            \|D^{\alpha}G(S_ju)\|_\infty \lesssim_{\alpha,G,\|u\|_\infty} 2^{|\alpha|j},\quad j\geq 0
        \end{equation*}
        where the implied constant is independent of $j$. In particular, $G(S_ju)$ is smooth for all $j\geq 0$.
    \end{theorem}    
    
    In the commutative case, this is a consequence of the inequality
    (from Proposition \ref{homogeneous_bound}.\eqref{fourier_truncation_bound}) 
    \begin{equation*}
        \|D^\alpha S_ju\|_\infty \lesssim 2^{j|\alpha|}
    \end{equation*}
    In the noncommutative case this approach does not work since the partial derivatives of $S_ju$ do not necessarily commute with each other.
    Instead, we first consider the case that $G(t) = e^{\ri\xi t}$, and then use the Fourier transform.
    \begin{lemma}
        Let $u = u^*\in L_\infty(\Rl^d_\theta)$. Then for all $\xi\in \Rl$ and all $j\geq 0$ the element $e^{\ri\xi S_j u}$ is smooth, and for all $\alpha\in \Ntrl^d$ we have
        \begin{equation*}
            \|D^\alpha e^{\ri\xi S_ju}\|_\infty \lesssim_{\alpha,\|u\|_\infty} 2^{|\alpha|j}(1+|\xi|)^{|\alpha|}
        \end{equation*}
        where the implied constant does not depend on $j$ or $\xi$.
    \end{lemma}
    \begin{proof}
        Since $S_ju$ is self-adjoint, $e^{\ri\xi S_ju}$ is unitary and hence the case $|\alpha|=0$ is trivial. We will therefore concentrate on $|\alpha|>0.$
        Consider the case where $|\alpha|=1$. Then $D^\alpha = D_k$ for some $k=1,\ldots, d$.          
        By Duhamel's formula,
        \begin{equation}\label{base_case}
            D_ke^{\ri\xi S_ju} = i\xi\int_0^1 e^{\ri\xi(1-\theta)S_ju}(D_k S_ju)e^{\ri\xi\theta S_ju}\,d\theta.
        \end{equation}
        The integral should be understood as a weak integral in $L_{\infty}(\Rl^d_\theta).$
        Using the triangle inequality, it follows that
        \begin{equation*}
            \|D_ke^{\ri\xi S_ju}\|_\infty \leq |\xi|\|D_kS_ju\|_\infty.
        \end{equation*}
        Proposition \ref{homogeneous_bound}.\eqref{fourier_truncation_bound} implies that $\|D_k S_ju\|_{\infty} \lesssim 2^{j}\|u\|_{\infty}$
        
        The cases $|\alpha|> 1$ follow from repeated application of Leibniz's rule and Duhamel's formula. To see $|\alpha|=2$, let $l=1,\ldots,d$, and apply the Leibniz rule to \eqref{base_case}
        and then Duhamel's formula in the following way
        \begin{align*}
            D_lD_ke^{\ri\xi S_ju} &= -\xi^2\iint_{[0,1]^2} e^{\ri\xi(1-\theta)(1-\eta)S_ju}(D_lS_ju)e^{\ri\xi(1-\theta)\eta S_ju}(D_kS_ju)e^{\ri\xi\theta S_ju}\,d\eta d\theta\\
                                &\quad + i\xi\int_0^1 e^{\ri\xi(1-\theta)S_ju} (D_lD_kS_ju) e^{\ri\xi\theta S_ju}\,d\theta\\
                                &\quad - \xi^2\iint_{[0,1]^2} e^{\ri\xi(1-\theta)S_ju}(D_kS_ju)e^{\ri\xi\theta(1-\eta)S_ju}(D_lS_ju)e^{\ri\xi\theta\eta S_ju}\,d\theta d\eta.
        \end{align*}
        In the general case, $D^{\alpha}e^{\ri \xi S_ju}$ is a linear combination of of integrals of expressions of the form
        \[
            \xi^{k}e^{\ri \xi\theta_0 S_ju}(D^{\alpha_0}S_ju)e^{\ri \xi \theta_1S_ju}(D^{\alpha_1}S_ju)\cdots (D^{\alpha_{n-1}}S_ju) e^{\ri \xi \theta_nS_ju}
        \]
        where $|\alpha_0|+\cdots+|\alpha_{n-1}| = |\alpha|$ and $0\leq k\leq |\alpha|.$ Applying Proposition \ref{homogeneous_bound}.\eqref{fourier_truncation_bound} to each integrand yields
        the desired bound for $\|D^{\alpha}e^{\ri \xi S_ju}\|_{\infty}.$
    \end{proof}

    \begin{proof}[Proof of Theorem \ref{smooth_coefficient}]
        Note that $\sup_{j\geq 0}\|S_ju\|_\infty < \infty$, and hence without loss of generality we may assume that $G$ is compactly supported.
        Since $G$ is a compactly supported smooth function, there exists the Fourier representation
        \begin{equation}\label{fourier_description}
            G(S_ju) = \int_{-\infty}^\infty \widehat{G}(\xi)\exp(2\pi \ri\xi S_ju)\,d\xi.
        \end{equation}
        The Fourier transform $\widehat{G}$ is in the Schwartz class $\Sc(\Rl)$. 
        Since $D^{\alpha}e^{2\pi\ri\xi S_ju} \in L_\infty(\Rl^d_\theta)$ and has norm with at most polynomial growth in $\xi$, it follows that the integral \eqref{fourier_description}
        converges. Thus,
        $$
            \|D^{\alpha}G(S_ju)\|_\infty \lesssim_{\|u\|_\infty} 2^{j|\alpha|}\int_{\Rl} \widehat{G}(\xi)(1+|\xi|)^{|\alpha|}\,d\xi < \infty.
        $$
        
    \end{proof}
    We now arrive at the main result concerning the stability of $B^s_{p,q}(\Rl^d_\theta)$ under Nemytskij operators. 
    This theorem is a noncommutative analogy of \cite[Theorem 2.87]{Bahouri-Chemin-Danchin-2011}.
    \begin{theorem}\label{nemytskij_mapping}
        Let $F \in C^\infty(\Rl)$, let $s > 0$ and $p,q \in [1,\infty]$. Then
        \begin{equation*}
            u = u^* \in B^s_{p,q}(\Rl^d_\theta)\cap L_\infty(\Rl^d_\theta)\Longrightarrow F(u)-F(0) \in B^s_{p,q}(\Rl^d_\theta)\cap L_\infty(\Rl^d_\theta).
        \end{equation*}
    \end{theorem}
    \begin{proof}
        Let $u=u^* \in B^s_{p,q}(\Rl^d_\theta)\cap L_\infty(\Rl^d_\theta)$ for some $s > 0$ and $p,q\in [1,\infty]$. 
        Using the ``Meyer" representation \eqref{meyer_decomposition}, for all $j\geq 0$ we have
        $$
            F(S_ju) = F(0)+G(S_0u)\Delta_0(u) + \int_{\Omega} T_{\omega}(S_ju)\,d\mu(\omega).
        $$
        As $j\to \infty$, the left hand side converges to $F(u)$ in the $L_p$ norm due to Corollary \ref{nonlinear_continuity}. As for the right hand side, we use Theorem \ref{besov_mapping_theorem},
        \begin{align*}
            \left\|\int_{\Omega} T_{\omega}(u-S_ju)\,d\mu(\omega)\right\|_{B^s_{p,q}} &\leq \int_{\Omega} \|T_{\omega}\|_{B^{s}_{p,q}\to B^s_{p,q}}\,d|\mu(\omega)|\|u-S_ju\|_{B^{s}_{p,q}}\\
                                                                                    &\leq \int_{\Omega} M_{s+2}(\alpha(\cdot,\omega))M_{s+2}(\beta(\cdot,\omega))\,d|\mu(\omega)|.
        \end{align*}
        By Theorem \ref{smooth_coefficient} and the definition of $\BS^\infty$, we have
        $$
            \int_{\Omega} M_{s+2}(\alpha(\cdot,\omega))M_{s+2}(\beta(\cdot,\omega))\,d|\mu(\omega)| < \infty.
        $$
        Therefore,
        \begin{equation*}
            \lim_{j\to\infty} \int_{\Omega} T_{\omega}(S_ju)\,d\mu(\omega) =  \int_{\Omega} T_{\omega}(u)\,d\mu(\omega)
        \end{equation*}
        in the $B^{s}_{p,q}$-topology. In particular, in the $L_p$-topology. 
        
        Similarly, the operator 
        \[
            v\mapsto G(S_0u)\Delta_0(v)
        \]
        is an elementary pseudodifferential operator, and hence $G(S_0u)\Delta_0u \in B^s_{p,q}(\Rl^d_\theta).$
        
        We arrive at the following:
        $$
            F(u) = F(0)+G(S_0u)\Delta_0u + \int_{\Omega} T_{\omega}(u)\,d\mu(\omega) \in F(0)+B^s_{p,q}(\Rl^d_\theta).
        $$
        Hence, $F(u)-F(0) \in B^s_{p,q}(\Rl^d_\theta)$, and since $u\in L_\infty(\Rl^d_\theta)$, Corollary \ref{nonlinear_continuity} implies that $F(u)-F(0) \in B^s_{p,q}(\Rl^d_\theta)\cap L_{\infty}(\Rl^d_\theta)$.
    \end{proof}
\subsection{Differences of Nemytskij operators}
    Let $u,v\in  L_\infty(\Rl^d_\theta)$ be self-adjoint. Let $F \in C^\infty(\Rl).$ Modifying $F$ outside a compact set
    if necessary, we assume that $F^{[1]} \in \BS^{\infty}$ and that
    \[
        \frac{F(t)-F(s)}{t-s} = \int_{\Omega} \alpha(t,\omega)\beta(s,\omega)\,d\mu(\omega).
    \]    
    By the L\"owner formula (Lemma \ref{lowner_formula}), we have
    \[
        F(u)-F(v) = \int_{\Omega} \alpha(u,\omega)(u-v)\beta(v,\omega)\,d\mu(\omega).
    \]
    We will prove that $F$ is locally Lipschitz on $B^s_{p,q}(\Rl^d_\theta)\cap L_{\infty}(\Rl^d_\theta)$ by an application of Corollary \ref{product_estimate} to the integrand. We have
    \[
        \|\alpha(u,\omega)(u-v)\beta(v,\omega)\|_{B^s_{p,q}\cap L_{\infty}} \leq \|\alpha(u,\omega)\|_{B^s_{p,q}\cap L_{\infty}}\|u-v\|_{B^s_{p,q}\cap L_{\infty}}\|\beta(v,\omega)\|_{B^s_{p,q}\cap L_{\infty}}.
    \]
    Formally, it follows that
    \[
        \|F(u)-F(v)\|_{B^s_{p,q}\cap L_{\infty}} \leq \|u-v\|_{B^s_{p,q}\cap L_{\infty}}\int_{\Omega} \|\alpha(u,\omega)\|_{B^s_{p,q}\cap L_{\infty}}\|\beta(v,\omega)\|_{B^s_{p,q}\cap L_{\infty}}\,d|\mu(\omega)|.
    \]  
    We will take some care to justify this inequality.

%
    The results of this section are summarised in the following theorem. (Compare \cite[Chapter 2, Section 7]{Taylor-tools-for-pde-2000}.)
    \begin{theorem}\label{smooth_is_locally_lipschitz}
        Let $s > 0$ and $p,q\in [1,\infty]$, and let $F\in C^\infty(\Rl)$. If $u,v \in L_\infty(\Rl^d)\cap B^s_{p,q}(\Rl^d_\theta)$
        are self-adjoint, then $F(u)-F(v) \in B^s_{p,q}(\Rl^d_\theta)$ with
        $$
            \|F(u)-F(v)\|_{B^s_{p,q}\cap L_\infty} \lesssim_{F,\|u\|_{B^s_{p,q}\cap L_{\infty}},\|v\|_{B^s_{p,q}\cap L_{\infty}}} \|u-v\|_{B^s_{p,q}\cap L_{\infty}}.
        $$
    \end{theorem}
    In particular, if $s > \frac{d}{p}$ then $F$ is locally Lipschitz on $B^s_{p,q}(\Rl^d_\theta).$ When $\det(\theta)\neq 0,$ then no restrictions
    on $s,p$ and $q$ are necessary.

\subsection{Functions of several variables}
    The classical theory of Nemytskij operators includes the study of functions of several variables,
    $$
        (u_1,\ldots,u_n) \mapsto f(u_1,\ldots,u_n),\quad f \in C^\infty(\Rl^n).
    $$
    See, for example, \cite[Section 5.5.1]{Runst-Sickel-1996}. It is not obvious how this can be generalised to the noncommutative setting.
    Therefore we will only deal with {polynomial} functions. A noncommutative polynomial in $n$ variables $(X_1,\ldots,X_n)$ is a formal expression
    $$
        f(X_1,X_2,\ldots,X_n) = \sum_{1\leq j_1,j_2,\ldots,j_l\leq n} a_{j_1,j_2,\ldots,j_l} X_{j_1}X_{j_2}\cdots X_{j_l}
    $$
    where each coefficient $a_{j_1,\ldots,j_l}$ is scalar, and the number of terms is finite. Given such a polynomial and $u_1,\ldots,u_n \in L_\infty(\Rl^d_\theta)$, we will denote
    $$
        f(u_1,u_2,\ldots,u_n) := \sum_{1\leq j_1,j_2,\ldots,j_l\leq n} a_{j_1,j_2,\ldots,j_j} u_{j_1}u_{j_2}\cdots u_{j_l}.
    $$
    To be precise, $f$ is an element of the free associative algebra in $n$ variables $X_1,\ldots,X_n$, and the evaluation $f \mapsto f(u_1,\ldots,u_n)$ is the homomorphism to $L_\infty(\Rl^d_\theta)$ determined 
    by mapping $X_j$ to $u_j$. The notation $f(u_1,\ldots,u_n)$ is meant to suggest functional calculus, but here $f$ is a polynomial in $n$ noncommuting variables and not a function on $\Rl^n.$
    
    Local Lipschitz estimates for polynomial functions on $B^s_{p,q}(\Rl^d_\theta)\cap L_\infty(\Rl^d_\theta)$ follow from Corollary \ref{product_estimate}, we give the details below.
    \begin{lemma}\label{polynomials_are_good}
%
        Let $f(X_1,X_2,\ldots,X_n)$ be a noncommutative polynomial in $n$ variables, let $p,q\in [1,\infty]$ and $s > 0$. The assignment
        $$
            (u_1,\ldots,u_n) \mapsto f(u_1,\ldots,u_n)
        $$
        is locally Lipschitz on $B^s_{p,q}(\Rl^d_\theta)\cap L_{\infty}(\Rl^d_\theta)$ in the sense that there is a norm bound
        $$
            \|f(u_1,\ldots,u_n)-f(v_1,\ldots,v_n)\|_{B^s_{p,q}(\Rl^d_\theta)\cap L_{\infty}(\Rl^d_\theta)} \leq C\max_{1\leq j\leq n} \|u_j-v_j\|_{B^s_{p,q}(\Rl^d_\theta)\cap L_\infty(\Rl^d_\theta)}.
        $$
        where the constant $C$ depends on $f$, and all of the norms
        $$
            \|u_j\|_{L_\infty(\Rl^d_\theta)\cap B^s_{p,q}(\Rl^d_\theta)},\|v_j\|_{L_{\infty}(\Rl^d_\theta)\cap B^s_{p,q}(\Rl^d_\theta)},\quad 1\leq j\leq n.
        $$
     \end{lemma}
     \begin{proof}
        Observe that due to linearity, it suffices to prove the assertion for monomials of the form
        $$
            f(X_1,\ldots,X_n) = X_{i_1}X_{i_2}\cdots X_{i_k}
        $$
        where $k\geq 1$ and the indices $1\leq i_1,\ldots,i_k \leq n$ are not necessarily distinct. We can refer to $f$ in this form as monomial of degree $k$. We prove the assertion for all monomials of degree $k$ by induction on $k$, with the $k=1$ case being trivial.
        Let $k\geq 1$, then we have,
        \begin{align*}
            f(u_1,\ldots,u_n)-f(v_1,\ldots,v_n) &= (u_{i_1}u_{i_2}\cdots u_{i_k}) - (v_{i_1}v_{i_2}\cdots v_{i_k})\\    
                                                &= (u_{i_1}-v_{i_1})(u_{i_2}\cdots u_{i_k}) + v_{i_1}((u_{i_2}\cdots u_{i_k})-(v_{i_2}\cdots v_{i_k})).
        \end{align*}    
        That is,
        \begin{equation*}
            f(u_1,\ldots,u_n)-f(v_1,\ldots,v_n) = (u_{i_1}-v_{i_1})g(u_1,\ldots,u_{n}) + v_{i_1}(g(u_1,\ldots,u_n)-g(v_1,\ldots,v_n))
        \end{equation*}
        where $g(X_1,\ldots,X_n) = X_{i_2}X_{i_3}\cdots X_{i_k}$ is a monomial of degree $k-1$. From Theorem \ref{product_estimate}, we have
        \begin{align*}
            \|f(u_1,\ldots,u_n)-f(v_1,\ldots,v_n)\|_{B^s_{p,q}(\Rl^d_\theta)} &\leq \|u_{i_1}-v_{i_1}\|_{B^s_{p,q}(\Rl^d_\theta)}\|g(u_1,\ldots,u_n)\|_{\infty}\\
                                                                              &\quad + \|u_{i_1}-v_{i_1}\|_{\infty}\|g(u_1,\ldots,u_n)\|_{B^s_{p,q}(\Rl^d_\theta)}\\
                                                                              &\quad + \|v_{i_1}\|_\infty\|g(u_1,\ldots,u_n)-g(v_1,\ldots,v_n)\|_{B^s_{p,q}(\Rl^d_\theta)}\\
                                                                              &\quad + \|v_{i_1}\|_{B^s_{p,q}(\Rl^d_\theta)}\|g(u_1,\ldots,u_n)-g(v_1,\ldots,v_n)\|_{\infty}.
        \end{align*}
        By the inductive hypothesis, the result follows.
     \end{proof}
    
\section{Nonlinear partial differential equations on $\Rl^d_\theta$}\label{pde_section}
    We will focus on the following three model classes non-linear evolution equations:
    \begin{enumerate}[{\rm (i)}]
        \item{} The nonlinear heat equation, also called a reaction-diffusion or Allen-Cahn equation
        \begin{equation}\label{NLHE}
            \partial_t u = \Delta u + F(u)
        \end{equation}
        where $F \in C^\infty(\Rl)$ is real-valued. We also consider the case where $F(u)$ is a noncommutative polynomial in $\{u,\partial_1u,\ldots,\partial_d u\}$
        \item{} The nonlinear Schr\"odinger equation
        \begin{equation}\label{NLS}
            \partial_t u = \ri\Delta u + g(u)
        \end{equation}
        where $g(u)$ is a polynomial in the variables $\{u,\partial_1 u,\partial_2 u,\ldots,\partial_d u\}$ and their adjoints. This is essentially a noncommutative version of the class of nonlinear Schr\"odinger equations studied in, for example \cite{GinibreVelo1979}, \cite{Cazenave-semilinear-2003} or \cite[Chapter 3]{Tao-nonlinear-dispersive-2006}.
        \item{} Finally, we will discuss ``fluid" equations. Specifically, the system of equations for $(u_1,u_2,\ldots,u_d)$ and $p$ given by
        \begin{equation*}
            \partial_t u_j + \frac{1}{2}\sum_{k=1}^d u_k(\partial_k u_j)+(\partial_k u_j)u_k = \Delta u_j + \partial_j p,\quad \sum_{j=1}^d \partial_j u_j = 0,\quad j=1,\ldots,d.
        \end{equation*}
    \end{enumerate}
    We will adopt the following notation for referring to time-dependent objects: given $u \in C([0,T],X)$, we write $u(t)$ for the value
    of $u$ at time $t$, and $\partial_t u$ denotes the derivative of $u$ with respect to $t$ in the sense that
    $$
        \lim_{h\to 0}\left\|\frac{u(t+h)-u(t)}{h}-\partial_tu(t)\right\|_X = 0.
    $$
    
    Our results are based on the following existence theorem, which is well known.
    Let $X$ be a Banach space, and let $Y$ be a Banach subspace of $X$. Without loss of generality, we may assume that the inclusion $Y\subseteq X$ is contractive. That is,
    $$
        \|x\|_{X} \leq \|x\|_{Y},\quad x \in X.
    $$
    A function $F:Y\to X$ is said to be locally Lipschitz if for all $R > 0$ and $x_1,x_2 \in Y$ with $\|x_1\|_Y, \|x_2\|_Y \leq R$ there is a constant $C_{F,R}$ such that
    $$
        \|F(x_1)-F(x_2)\|_{X} \leq C_{F,R}\|x_1-x_2\|_Y.
    $$
    For a Banach space $X$, we denote by $C([0,T],X)$ the Banach space of continuous $X$-valued functions on the interval $[0,T]$
    with norm $\|f\|_{C([0,T],X)} = \sup_{0\leq s \leq T} \|f(s)\|_X$.
    
    The following theorem is \cite[Chapter 15, Proposition 1.1]{Taylor-pde-3-2011}.
    \begin{theorem}\label{abstract_existence_theorem}
        Let $X$ and $Y$ be as above. Assume that $L$ is a closed densely defined operator on $X$, let $F:Y\to X$ be a function. Assume that the following three conditions hold
        \begin{enumerate}[{\rm (i)}]
            \item{} $L$ generates a contractive $C_0$-semigroup $e^{tL}$ on $Y$,
            \item{} For all $t > 0$, we have
            $$
                e^{tL}:X\to Y
            $$
            continuously,
            \item{} $F$ is locally Lipschitz from $Y$ to $X$,
            \item{} The norm function $t\mapsto \|e^{tL}\|_{X\to Y}$ obeys
            $$
                \|e^{tL}\|_{X\to Y} \leq Ct^{-\gamma}
            $$
            where $\gamma<1,$ for sufficiently small $t,$ and for some constant $C$ independent of $t.$
        \end{enumerate}
        Then for all $x_0 \in Y$ there exists a maximal $T_{x_0} > 0$ and a unique
        $$
            x \in C([0,T_{x_0}),Y)
        $$
        such that
        $$
            x(t) = e^{tL}x_0 + \int_0^t e^{(t-s)L}F(x(s))\,ds,\quad 0\leq t < T_{u_0}
        $$
        where the integral converges in the $Y$-valued Bochner sense.
    \end{theorem}  
    
    The following simple estimate is also useful.
    \begin{lemma}\label{convolution_lemma}
        Let $L$ be as in Theorem \ref{abstract_existence_theorem}. Let $T > 0$, and assume that
        $$
            f \in C([0,T],X).
        $$
        Then for all $t > 0$ the integral
        $$
            \Psi(t) = \int_{0}^t e^{(t-s)L}f(s)\,ds
        $$
        converges in the $Y$-valued Bochner sense, and with $\Psi(0) = 0$ defines an element
        $$
            \Psi \in C([0,T],Y).
        $$
        with
        $$
            \|\Psi\|_{C([0,T],Y)} \leq \int_0^{T} \|e^{(t-s)L}\|_{X\to Y}\,ds\|f\|_{C([0,T],X)}.
        $$
    \end{lemma}
       
    We provide the following global existence criterion, which is not difficult to prove directly but for which we have not been able to find a precise reference. A proof is given in the appendix.
    \begin{theorem}\label{abstract_persistence_theorem}
        Adopt the notation of Theorem \ref{abstract_existence_theorem}. If $\{x(t)\}_{0 \leq t < T_{x_0}}$
        is the fixed point, and
        $$
            \sup_{0 \leq t < T_{x_0}} \frac{\|F(x(s))\|_X}{\|x(s)\|_Y} < \infty
        $$
        then $T_{x_0} = \infty$.
    \end{theorem}
    
    In order to upgrade mild solutions to classical solutions, we use \cite[Proposition 4.1.6]{Cazenave-Haraux-semilinear-evolution-equations-1998},
    which implies that it suffices that $F(u(t)) \in L_1((0,T),\mathrm{dom}(L)).$

\subsection{Allen-Cahn equations}
    We study now the nonlinear heat equation \eqref{NLHE}. 
    
    For generic smooth nonlinearities $F$, we can prove local well-posedness for initial data in the Besov class $B^s_{\infty,\infty}(\Rl^d_\theta)$. This
    is the noncommutative analogy of \cite[Section 7.3.1]{Lunardi-parabolic-1995}.

    \begin{theorem}\label{nlhe_lwp}
        Let $F \in C^\infty(\Rl)$ be real-valued, and let $r>0.$ For all $u_0 = u_0^* \in B^r_{\infty,\infty}(\Rl^d_\theta),$ there exists $T_{u_0}>0$ and a unique
        $$
            u = (t\mapsto u(t)) \in C([0,T_{u_0}),B^r_{\infty,\infty}(\Rl^d_\theta))\cap \bigcap_{\alpha>0} C^1((0,T_{u_0}),B^\alpha_{\infty,\infty}(\Rl^d_\theta))
        $$
        such that $u(0) = u_0$, and
        $$
            \partial_tu(t) = \Delta u(t) + F(u(t))
        $$
        for all $0 < t < T_{u_0}.$ 
    \end{theorem}
    \begin{proof}
        We appeal to the conditions of Theorem \ref{abstract_existence_theorem} to prove the existence of $T_{u_0} > 0$ and a unique $u \in C([0,T_{u_0}),B^r_{\infty,\infty}(\Rl^d_\theta))$
        such that
        \begin{equation}\label{nlhe_mild}
            u(t) = e^{t\Delta}u_0 + \int_{0}^t e^{(t-s)\Delta}F(u(s))\,ds,\quad 0\leq t < T_{u_0}.
        \end{equation}
        Indeed, we take $X=Y=B^r_{\infty,\infty}(\Rl^d_\theta)$ and apply Theorem \ref{smooth_is_locally_lipschitz} with $p=q=\infty$ to ensure the local Lipschitz condition
        on $F$.
        
        Now we prove that $u(t) \in \bigcap_{u> 0}B^u_{\infty,\infty}(\Rl^d_\theta)$ for all $t > 0$. We can do this by induction, proving that if $\alpha > 0$ is such that $u(s) \in B^\alpha_{\infty,\infty}(\Rl^d_\theta)$
        for all $s \leq t$ then $u(s) \in B^{\alpha+\frac{1}{2}}_{\infty,\infty}(\Rl^d_\theta)$ for all $0\leq s\leq t.$ 
        Suppose that $t > 0$ is such that $u(s)\in B^\alpha_{\infty,\infty}(\Rl^d_\theta)$ for all $0\leq s \leq t$. 
        Theorem \ref{smooth_is_locally_lipschitz} implies that the function $t\mapsto F(u(t))$ belongs to $C([0,t],B^\alpha_{\infty,\infty}(\Rl^d_\theta))$. 
        
        From Proposition \ref{heat_mapping_besov}, we have
        $$
            \|e^{t\Delta}\|_{B^\alpha_{\infty,\infty}(\Rl^d_\theta)\to B^{\alpha+\frac{1}{2}}(\Rl^d_\theta)} \lesssim_r 1+t^{-\frac{1}{2}}.
        $$
        Hence $t\mapsto \|e^{t\Delta}\|_{B^\alpha_{\infty,\infty}\to B^{\alpha+\frac{1}{2}}_{\infty,\infty}}$ is integrable near zero.
        This completes the verification of the conditions of Lemma \ref{convolution_lemma} with $X = B^{\alpha}_{\infty,\infty}(\Rl^d_\theta)$, $Y = B^{\alpha+\frac12}_{\infty,\infty}(\Rl^d_\theta)$, $L=\Delta$
        and $f(s) = F(u(s))$, and it follows from \eqref{nlhe_mild} that $u(t) \in B^{\alpha+\frac{1}{2}}_{\infty,\infty}(\Rl^d_\theta)$.
        
        Hence, $(t\mapsto u(t))\in \bigcap_{\alpha>0} C([0,T_{u_0}),B^{\alpha}_{\infty,\infty}(\Rl^d_\theta)).$
        The same is true for $t\mapsto F(u(t)),$ by Theorem \ref{nemytskij_mapping}.        
        It follows that $u'(t) = \Delta u(t)+F(u(t)),$ by the abstract result \cite[Proposition 4.1.6]{Cazenave-Haraux-semilinear-evolution-equations-1998}.
        
    \end{proof}
    
    We can give a similar theorem where the nonlinearity is a polynomial in $u$ and its derivatives.
    \begin{theorem}\label{parabolic_polynomial}
        Let $p(u,\partial_1 u,\ldots,\partial_d u)$ be a noncommutative polynomial in the variables $\{u,\partial_1 u,\ldots,\partial_d u\}$, and let $r > 1$.
        For all $u_0 \in B^r_{\infty,\infty}(\Rl^d_\theta)$ there exists a maximal $T_{u_0} > 0$ and a unique 
        $$
            u = (t\mapsto u(t)) \in C([0,T_{u_0}),B^r_{\infty,\infty}(\Rl^d_\theta))\cap \bigcap_{\alpha>0} C^1((0,T_{u_0}), B^\alpha_{\infty,\infty}(\Rl^d_\theta))
        $$
        such that $u(t) = u_0$ and
        $$
            \partial_t u(t) = \Delta u(t) + p(u(t),\partial_1 u(t),\ldots,\partial_d u(t)),\quad 0 < t < T_{u_0}.
        $$
    \end{theorem}
    \begin{proof}
        This proof is essentially identical to that of Theorem \ref{nlhe_lwp}. The only difference is that we use Theorem \ref{polynomials_are_good} in place of Theorem \ref{smooth_is_locally_lipschitz}.
    \end{proof}

\subsection{Nonlinear Schr\"odinger equations}
    We now study the nonlinear Schr\"odinger equation \eqref{NLS},
    $$
        \ri\frac{\partial u}{\partial t} = \Delta u + g(u,u^*)
    $$
    where $g(u,u^*)$ is a noncommutative polynomial in $u$ and $u^*$. Formally, solutions can be written in terms of the Schr\"odinger semigroup
    as
    \begin{equation}\label{mild_nlse}
        u(t) = e^{-\ri t\Delta}u(0) -\ri \int_0^t e^{-\ri (t-s)\Delta}g(u(s),u(s)^*)\,ds.
    \end{equation}
    
    Unlike the heat semigroup, the Schr\"odinger semigroup does not introduce any smoothing in the scale of Sobolev spaces. Local smoothing
    for the Schr\"odinger semigroup on $\Rl^d$ are known, but we postpone the task of proving analogies for the Schr\"odinger semigroup on $\Rl^d_\theta$ for future work.
    
    \begin{theorem}\label{schr_local}
        Assume that $\det(\theta)\neq 0.$
        Let $g(u,u^*)$ be a noncommutative polynomial in the variables $u$ and $u^*$. 
        Let $r\geq 0.$ For all $u_0 \in W^r_{2}(\Rl^d_\theta)$ there exists
        a maximal $T_{u_0} > 0$ and a unique
        $$
            u \in C([0,T_{u_0}),W^r_2(\Rl^d_\theta))
        $$
        such that $u$ obeys \eqref{mild_nlse} $\lim_{t\to 0}u(t)=  u_0$ in the $W^r_2(\Rl^d_\theta)$-sense.
    \end{theorem}
    \begin{proof}
        By Theorem \ref{schr_mapping}, the semigroup $t\mapsto \exp(-\ri t\Delta)$ is strongly continuous on $W^s_{2}(\Rl^d_\theta).$
        
        Since $\det(\theta)\neq 0,$ we have
        \[
            L_{\infty}(\Rl^d_\theta)\subset L_2(\Rl^d_\theta) \subset B^{r}_{2,2}(\Rl^d_\theta) = W^r_{2}(\Rl^d_\theta).
        \]
        Hence, by Theorem \ref{polynomials_are_good}, $u\mapsto g(u,u^*)$ is locally Lipschitz on $W^r_2(\Rl^d_\theta).$
        The result now follows from Theorem \ref{abstract_existence_theorem}.
    \end{proof}

    
    We can begin to see the counterintuitive consequences of Theorem \ref{schr_local} if we assume that $g$ has some special structure so that \eqref{mild_nlse} preserves the $L_2$-norm.
    
    \begin{lemma}
        Let $\det(\theta)\neq 0,$ and let $g(u,u^*)$ be a noncommutative polynomial in the variables $u$ and $u^*.$ Let $r\geq 0$ and $u_0 \in W^r_2(\Rl^d_\theta)$, and
        let
        \[
            u \in C([0,T_{u_0}),W^{r}_2(\Rl^d_\theta))
        \]
        be the solution to \eqref{mild_nlse} extended to a maximal interval $[0,T_{u_0}).$ If
        \[
            \|u(t)\|_{L_2(\Rl^d_\theta)} \leq \|u_0\|_{L_2(\Rl^d_\theta)},\quad 0\leq t < T_{u_0}
        \]
        then $T_{u_0} = \infty.$
    \end{lemma}

    Now we can begin to see the counterintuitive results of concentrating on $\det(\theta)\neq 0$.
    \begin{lemma}
        Let $\det(\theta)\neq 0$, and let $p>1$ be an odd positive integer, and let $\mu \in \Rl$. Then the nonlinear Schr\"odinger equation
        \begin{equation}\label{NLSp}
            \ri\partial_t u + \Delta u = \mu u|u|^{p-1}
        \end{equation}
        has unique global mild solutions for any initial $u(0) \in W^2_2(\Rl^d_\theta)$, which depend continuously on the initial conditions in the sense that if $t$ is sufficiently small then the function $u(0)\mapsto u(t)$
        is locally Lipschitz on $L_2(\Rl^d_\theta)$.
     \end{lemma}
     \begin{proof}
        This is essentially an application of \cite[Theorem 3.3.1]{Cazenave-semilinear-2003}.
        Since $p$ is an odd integer, $\mu u|u|^{p-1}$ is a noncommutative polynomial in $u$ and $u^*.$ Hence, Theorem \ref{schr_local} implies that there is a time interval $[0,T_{u_0})$
        such that \eqref{NLSp} has a solution in the mild sense in $L_2(\Rl^d_\theta)$ for $0\leq t < T_{u_0}.$ This mild solution obeys \eqref{NLSp} since $t\mapsto \mu u(t)|u(t)|^{p-1}$ is a continuous $W^{2}_2(\Rl^d_\theta)$-valued function, so that the conditions of \cite[Proposition 4.1.6]{Cazenave-Haraux-semilinear-evolution-equations-1998} are satisfied.
        The $L_2$ norm of $u(t)$ varies with $t$ as,
        \begin{align*}
              \frac{d}{dt}\|u(t)\|_{2}^2 &= \frac{d}{dt}\langle u(t),u(t)\rangle\\
                                         &= 2\rl\langle \partial_t u(t),u(t)\rangle\\
                                         &= 2\rl\langle \ri\Delta u(t)-\ri\mu u(t)|u(t)|^{p-1},u(t)  \rangle\\
                                         &= 2\rl\langle \ri\Delta u(t),u(t)\rangle- 2\mu\rl(\ri\tau(|u(t)|^{p+1}))\\
                                         &= 0.
        \end{align*}
        Note that in the final step we have used the fact that $\tau(|u(t)|^{p+1}) < \infty$, which follows from $\|u(t)\|_2 < \infty$ since $\det(\theta)\neq 0$.
        Hence, the $L_2$ norm of $u$ is conserved. It follows from \eqref{polynomials_are_good} that
        \begin{equation*}
            \sup_{0\leq s\leq T} \|\mu u|u|^{p-1}\|_{2} \lesssim_{\theta} \sup_{0\leq s\leq T} |\mu|\|u(s)\|_{2}^p \leq |\mu|\|u(0)\|_{2}^{p-1}\sup_{0\leq s\leq T}\|u(s)\|_{2}.
        \end{equation*}
        Hence Lemma \ref{abstract_persistence_theorem} implies that the maximal time of existence $T_{u_0}$ is infinite.
     \end{proof}
     Via an identical argument, if $g(u,u^*)$ is a noncommutative polynomial in $\{u,u^*\}$ such that $\langle g(u,u^*),u\rangle \in \Rl$ for all $u \in L_2(\Rl^d_\theta)$, it
     follows that \eqref{NLSp} has global-in-time solutions. For example, we may take $g(u) = |u|^{p-1}u$.

\subsection{Navier-Stokes equations}\label{fluids_section}
    As illustration of the differences between the commutative and noncommutative cases, we now discuss well-posedness for ``fluid" equations. 
    In particular we study the equation for self-adjoint $u(t) \in L_2(\Rl^d_\theta)\otimes \Rl^d$ and $p(t) \in L_2(\Rl^d_\theta)$,
    $$
        \partial_t u(t) + X(u(t)) = \Delta u(t) +\nabla p(t),\quad \nabla \cdot u(t) = 0,
    $$
    where $\nabla\cdot u = \sum_{j=1}^d \partial_j u$ and $X(u) = (X(u)_1,\ldots,X(u)_d)$ is the nonlinear term
    $$
        X(u)_k = \sum_{j=1}^d \frac{1}{2}(u_j\partial_j (u_k)+\partial_j(u_k)u_j),\quad k=1,\ldots,d.
    $$
    In other words, we have replaced the pointwise product $u_j\partial_j u_k$ in \eqref{advective_term} with the Jordan product
    $$
        u_j\circ \partial_j u_k := \frac{1}{2}(u_j \partial_j u_k + \partial_j u_k  u_j).
    $$
    Of course this is not the only polynomial which reduces to $u_j \partial_j u_k$ in the commutative case.
    We have selected $X(u)$ since it is arguably the simplest noncommutative analogy of $u\cdot\nabla u$ with the property that $X(u)$ is self-adjoint whenever $u$ is self-adjoint.
    We make no attempt to motivate this choice of analogy for the Navier-Stokes equations from a physical perspective. Instead, the main interest is simply the formal
    resemblance to the classical viscous incompressible Navier-Stokes equations.
    
    For the sake of brevity, denote
    $$
        H^s := W^s_2(\Rl^d_\theta)\otimes \Cplx^d,\quad s \in \Rl.
    $$
    
    As in the classical case, we dispense with the $\nabla p$ term by applying the orthogonal projection onto the divergence-free vector fields. 
    \begin{definition}
        The Leray projection $\Pl$ is the linear operator on $H^0$ defined in terms of map $\lt$ as
        \begin{equation*}
            (\Pl \lt(f))_j = \sum_{k=1}^d \lt((\delta_{j,k}-\frac{t_jt_k}{|t|^2})f(t)),\quad f \in L_2(\Rl^d).
        \end{equation*}
        That is, $\Pl$ is the multiplier $m(D) = \{m_{j,k}(D)\}$ where
        $$
            m_{j,k}(\xi) = \delta_{j,k} - \frac{\xi_j\xi_k}{|\xi|^2},\quad \xi \in \Rl^d\setminus \{0\}.
        $$
    \end{definition}
    Equivalently, we could define $\Pl$ to be the image under the Weyl transform $\wl$ of the classical Leray projection. 
    
    The Leray projection is bounded from $L_2(\Rl^d_\theta)\otimes \Cplx^d$ to itself, and has the property that if $u \in H^s$ for $s \geq 1$, then $\nabla\cdot \Pl u =0$.
    Alternatively, $\Pl$ could be defined as the $L_2$-orthogonal projection onto the closed subspace of divergence free vector fields. In addition, $\Pl$
    is bounded from $L_p(\Rl^d_\theta)\otimes \Cplx^d$ to itself for all $1<p<\infty$ due to the noncommutative Mikhlin theorem \cite{MSX3}, but we will not use this fact.

    \begin{lemma}\label{polynomials_are_good_hilbert}
        Let $\det(\theta)\neq 0$ and $s \geq 0$. Then the function
        \begin{equation*}
            u\mapsto \Pl X(u),\quad u \in H^{s+1}
        \end{equation*}
        is locally Lipschitz from $H^{s+1}$ to $H^{s}$. We also have the norm bound
        \begin{equation*}
            \|\Pl X(u)\|_{H^s} \lesssim_{\theta} \|u\|_{H^{s+1}}^2.
        \end{equation*}
    \end{lemma}
    \begin{proof}
        This follows from the boundedness of $\Pl$ on $H^s(\Rl^d_\theta),$ Theorem \ref{polynomials_are_good} and the fact that $L_2(\Rl^d_\theta)$ is contained in $L_{\infty}(\Rl^d_\theta)$
        for non-degenerate $\theta.$
    \end{proof}
%

     
     \begin{theorem}
        Assume that $\det(\theta)\neq 0$. Let $u_0 \in H^2$ have divergence zero and be self-adjoint. There exists a unique $u \in C^\infty((0,\infty),H^2)$ such that
        \begin{equation}\label{nc_leray_form}
            \partial_t u(t)+\Pl X(u(t)) = \Delta u(t),\quad \lim_{t\to 0} u(t) = u_0.
        \end{equation}
        Moreover, $u(t) = u(t)^*$ and $u(t) \in \bigcap_{s > 0} H^s$ for all $t > 0$.
     \end{theorem}
     \begin{proof}
        Lemma \ref{polynomials_are_good_hilbert} implies that $\Pl X$ is locally Lipschitz from
        $H^1$ to $H^0$. Hence, Theorem \ref{abstract_existence_theorem} implies that there exists a maximal time $T_{u(0)} \leq \infty$ such that a unique solution $u(t)$ in $H^1$ exists in the mild sense for all $0\leq t < T_{u(0)}$ and $u(0)=u_0 \in H^1.$ A similar argument to Theorem \ref{nlhe_lwp} shows that $u(t) \in \bigcap_{r>0} H^r$ for all $t>0.$ It follows that the  mild solution is a classical solution to \eqref{nc_leray_form} due to \cite[Proposition 4.1.6]{Cazenave-Haraux-semilinear-evolution-equations-1998}.
        Let $v(t) = \sum_{k=1}^d \partial_ku_k(t)$ be the divergence of $u$ at time $0\leq t < T_{u(0)}$. By assumption, we have $v(0) = 0$. Since $\nabla\cdot \Pl X(u(t)) = 0$
        by the definition of the Leray projection, we have
        \begin{equation*}
            \partial_t v(t) = \Delta v(t),\quad v(0) = 0.
        \end{equation*}
        Uniqueness for the heat equation implies that $v(t) = 0$ for all $t\geq 0$, and hence $u(t)$ has divergence zero for all $0\leq t < T_{u(0)}$.
        
        Observe that $X(u^*) = X(u)^*$ for all $u \in H^1$, and hence if $u(0) = u(0)^*$ it follows that $t\mapsto u(t)^*$ is a solution to \eqref{nc_leray_form}
        for all $0\leq t < T_{u(0)}$, which by uniqueness must be identical to $u(t)$. Hence $u(t) = u(t)^*$ for all $0 < t < T_{u_0}$.
        
        The $H^0$ norm of $u$ varies with time as
        \begin{equation*}
            \partial_t \|u(t)\|_{H^0}^2 = 2\rl\langle \Delta u(t)+\Pl X(u(t)),u(t)\rangle \leq 2\rl\langle \Pl X(u(t)),u(t)\rangle.
        \end{equation*}
        Since $u(t)$ has divergence zero for all $0 \leq t < T_{u_0}$, we have
        \begin{align*}
            \langle \Pl X(u(t)),u(t)\rangle &= \frac{1}{2}\sum_{k=1}^d \sum_{j=1}^d \langle\partial_j(u_ju_k+u_ku_j),u_k\rangle\\
                                        &= \frac12\sum_{j,k=1}^d \langle u_j(\partial_ju_k)+(\partial_j u_k)u_j,u_k\rangle\\
                                        &= \frac12\sum_{j,k=1}^d \tau((\partial_j u_k)^*u_j^*u_k)+\tau(u_j^*(\partial_ju_k)^*u_k)\\
                                        &= \frac12\sum_{j,k=1}^d \tau(u_k(\partial_j u_k)^*u_j^*)+\tau(u_j^*(\partial_ju_k)^*u_k)\\
                                        &= \frac12\sum_{j,k=1}^d \tau(u_j^*(u_k(\partial_j u_k)^*+(\partial_j u_k)^*u_k))\\
                                        &= \frac12\sum_{j,k=1}^d \langle u_j,u_k\partial_j (u_k^*)+\partial_j (u_k^*)u_k\rangle.
        \end{align*}
        Since $u_k(t)^* = u_k(t)$ for all $0\leq t < T_{u(0)}$, it follows that
        \begin{equation*} 
            \langle \Pl X(u(t)),u(t) \rangle = \frac12\sum_{j,k=1}^d \langle u_j,\partial_j(u_k^2)\rangle = -\frac12\langle \nabla\cdot u,\sum_{k=1}^d u_k^2\rangle = 0.
        \end{equation*}
        Hence, 
        \begin{equation*}
            \partial_t \|u(t)\|_{2}^2 \leq 0
        \end{equation*}
        so that the $H^0$ norm of $u(t)$ remains bounded over $0 \leq t < T_{u(0)}$. Theorem \ref{abstract_existence_theorem} now suffices to conclude that $T_{u_0}=\infty.$
        Indeed, we have
        \begin{equation*}
            \|X(u(t))\|_{H^0} \lesssim_\theta \|u(t)\|_{H^0}\|u(t)\|_{H^1}\leq \|u(0)\|_{H^0}\|u(t)\|_{H^1},\quad 0 < t < T_{u_0},
        \end{equation*}
        so that the condition of Theorem \ref{abstract_persistence_theorem} is satisfied, and hence $T_{u_0} = \infty$.
     \end{proof}
     
     \begin{remark}
        As the preceding argument shows, the well-posedness of \eqref{nc_leray_form} for $u_0 \in H^1$ is rather trivial. A more difficult problem would be to replace the nonlinear term
        with an expression that is not self-adjoint. For example, instead of
        \[
            X(u)_k = \sum_{j=1}^d \frac{1}{2}(u_j\partial_j (u_k)+\partial_j(u_k)u_j),\quad k=1,\ldots,d.
        \]
        we could examine
        \[
            X(u)_k = \sum_{j=1}^d u_j\partial_j (u_k),\quad k=1,\ldots,d.
        \]
        This problem would be more difficult because $X(u)_k^* \neq X(u^*)_k$ in general, and therefore we should not expect that if the components of $u_0$ are self-adjoint then
        the same is true for $u(t)$ for $t>0.$ 
     \end{remark}   
%
%

\appendix

\section{Existence theorems for parabolic differential equations}
    In this paper we have used the theory of abstract Cauchy problems to deal with existence and uniqueness of partial differential equations
    on $\Rl^d_\theta$. The general theory of abstract Cauchy problems is covered in, e.g., \cite[Chapter 4]{Lunardi-parabolic-1995}, \cite[Section 1.2]{Amann-parabolic-1995}, \cite[Chapter 3]{ArendtBattyHieberNeubrander2011}, \cite[Chapter 5]{PrussSimonett2016}, \cite[Appendix G]{HvNVW2017}, \cite{Yagi2010}, \cite[Chapter 7]{EngelNagel2000}.
    
%
    We have made use of a ``blow-up" criterion (Theorem \ref{abstract_persistence_theorem}) which is certainly not novel but for which we have been unable to find a reference. 

    \begin{proof}[Proof of Theorem \ref{abstract_persistence_theorem}]
        Suppose that $T_{x_0} < \infty$. If this is the case, then we must have $\|x(s)\|_{Y} \to\infty$ as $s\to T_{x_0}$. Indeed, otherwise there
        would exist $R > 0$ and a sequence $\{\varepsilon_j\}_{j=0}^\infty$ with $\varepsilon_j\to 0$ such that $\|x_{T_{u_0}-\varepsilon_j}\|_Y \leq R$. 
        According to Theorem \ref{abstract_existence_theorem}, there exists $T_R$ such that the solution can then be extended to an interval $[0,T_{u_0}-\varepsilon_j+T_R)$.
        This contradicts the maximality of $T_{u_0}$ when $j$ is large enough, and hence we must have $\|x(s)\|_{Y}\to\infty$.
    
        Now assume that there exists  $0 < C_{u_0} < \infty$ such that
        $$
            \|F(x(s))\|_{X} \leq C_{u_0}\|x(s)\|_Y,\quad 0\leq s < T_{x_0}.
        $$
        Let $0 < r < t < T_{u_0}$. Then by definition, we have
        \begin{equation*}
            x(t) = e^{(t-r)L}x(r) + \int_r^t e^{(t-s)L}F(x(s))\,ds.
        \end{equation*}
        Using the contractivity of the semigroup $e^{tL}$ and Lemma \ref{convolution_lemma}, it follows that
        \begin{equation*}
            \|x(t)\|_Y \leq \|x(r)\|_Y+\int_0^{t-r} \|e^{sL}\|_{X\to Y} \,ds\cdot  C_{u_0} \sup_{r\leq s\leq t} \|x(s)\|_{Y}
        \end{equation*}
        Thus, provided $t-r$ is sufficiently small, we have
        \begin{equation*}
            \sup_{r\leq s\leq t}\|x(s)\|_Y \leq \frac{\|x(r)\|_Y}{1-C_{u_0}\int_{0}^{t-r} \|e^{sL}\|_{X\to Y}\,ds}.
        \end{equation*}
        This contradicts $T_{u_0} < \infty$, since otherwise we could take $\varepsilon < t-r$ and $\|x(s)\|_Y$
        would not blow up in the interval $(T_{u_0}-\varepsilon,T_{u_0})$.
    \end{proof}

\end{document}